\documentclass[a4paper,reqno,11pt]{amsart}
\usepackage[headings]{fullpage}
\usepackage{amsmath} \usepackage{amssymb} \usepackage{amsopn}
\usepackage{amsthm} \usepackage{amsfonts} \usepackage{mathrsfs}
\usepackage{mathtools}
\usepackage{stmaryrd}
\usepackage{manfnt} \usepackage{mathdots}
\usepackage[OT2, T1]{fontenc}
\usepackage[dvips]{graphicx} \usepackage[dvipsnames,usenames]{color}
\usepackage[all]{xy}
\usepackage{tikz} \usetikzlibrary{decorations.markings,decorations.pathmorphing,matrix,arrows,chains,calc}
\usepackage{dsfont}
\usepackage[latin1]{inputenc}  
\usepackage{ifthen}
\usepackage{marvosym}
\usepackage{wasysym}
\usepackage{marginnote}
\usepackage[pdfpagelabels, pdftex,bookmarksopen,bookmarksdepth=2]{hyperref}
\hypersetup{
  pdftitle={},
  pdfauthor={},
  pdfsubject={},
  pdfkeywords={},
  colorlinks=true,    
  linkcolor=BrickRed,
  citecolor=OliveGreen,     
  filecolor=Blue,      
  urlcolor=Blue,       
  breaklinks=true,
  bookmarksopen=true,
  bookmarksnumbered=true,
  pdfpagemode=UseOutlines,
  plainpages=false}
\DeclareMathOperator{\rM}{M}

\DeclareMathOperator{\rc}{c}

\DeclareMathOperator{\rn}{n}

\newcommand{\bA}{{\mathbb A}}

\newcommand{\bE}{{\mathbb E}}
\newcommand{\bF}{{\mathbb F}}
\newcommand{\bG}{{\mathbb G}}
\newcommand{\bH}{{\mathbb H}}

\newcommand{\bN}{{\mathbb N}}

\newcommand{\bP}{{\mathbb P}}
\newcommand{\bQ}{{\mathbb Q}}
\newcommand{\bR}{{\mathbb R}}
\newcommand{\bS}{{\mathbb S}}

\newcommand{\bV}{{\mathbb V}}

\newcommand{\bZ}{{\mathbb Z}}

\newcommand{\cO}{{\mathscr O}}

\newcommand{\cX}{{\mathscr X}}
\newcommand{\cY}{{\mathscr Y}}

\newcommand{\dA}{{\mathcal A}}
\newcommand{\dB}{{\mathcal B}}

\newcommand{\fO}{{\mathfrak O}}

\newcommand{\fp}{{\mathfrak p}}

\DeclareSymbolFont{cyrletters}{OT2}{wncyr}{m}{n}
\DeclareMathSymbol{\Sha}{\mathalpha}{cyrletters}{"58}

\DeclareMathOperator{\Aut}{Aut}

\DeclareMathOperator{\Hom}{Hom}
\DeclareMathOperator{\id}{id}

\newcommand{\longhookrightarrow}{\lhook\joinrel\longrightarrow}

\DeclareMathOperator{\Isom}{Isom}
\newcommand{\longto}{\longrightarrow}
\DeclareMathOperator{\Maps}{Maps}

\newcommand{\surj}{\twoheadrightarrow} 
\newcommand{\longtwoheadrightarrow}{\relbar\joinrel\twoheadrightarrow}

\newcommand{\xyinj}{\ar@{^(->}}

\DeclareMathOperator{\GL}{GL}

\DeclareMathOperator{\PGL}{PGL}
\DeclareMathOperator{\PSL}{PSL}

\DeclareMathOperator{\SL}{SL}

\newcommand{\tr}{{\rm tr}} 
\newcommand{\matzz}[4]{\left(\!\!
\begin{array}{cc} #1 & #2 \\ #3 & #4 \end{array} \!\! \right)}

\newcommand{\vekz}[2]{\left(\!\!
    \begin{array}{c} #1 \\ #2  \end{array} \!\! \right)}

\DeclareMathOperator{\Alb}{Alb}

\DeclareMathOperator{\Proj}{Proj}

\DeclareMathOperator{\Quot}{Quot}

\DeclareMathOperator{\Spec}{Spec}
\DeclareMathOperator{\Spf}{Spf}

\DeclareMathOperator{\Sing}{Sing}

\newcommand{\prr}[1]{(\!(#1)\!)}
\newcommand{\pbb}[1]{[\![#1]\!]}

\DeclareMathOperator{\Br}{Br}

\newcommand{\Ram}{{\rm Ram}}

\DeclareMathOperator{\res}{res}

\def\10{{\overrightarrow{10}}}
\def\01{{\overrightarrow{01}}}
\newcommand{\ab}{{\rm ab}}

\newcommand{\ep}{\varepsilon}

\newcommand{\et}{\text{\rm \'et}}

\newcommand{\ket}{\text{\rm k\'et}}

\newcommand{\orb}{{\rm orb}}

\DeclareMathOperator{\Sym}{Sym}

\newtheorem{thm}{Theorem}[section]

\newtheorem{prop}[thm]{Proposition}
\newtheorem{lem}[thm]{Lemma}

\newtheorem{cor}[thm]{Corollary}

\newtheorem{thmABC}{Theorem}

\theoremstyle{definition}
\newtheorem{defi}[thm]{Definition}

\theoremstyle{remark}
\newtheorem{rmk}[thm]{Remark}

\newtheorem{nota}[thm]{Notation}

\newenvironment{pro*}[1][\proofname]{{\it{#1:}} }{}
\newenvironment{pro**}[1][]{{\it{#1}} }{\hfill $\square$}

\numberwithin{equation}{section}

\newcounter{step}[thm]

\tikzset{->-/.style={decoration={markings, mark=at position #1 with {\arrow{>}}},postaction={decorate}}}
\tikzset{->>-/.style={decoration={markings, mark=at position 0.55 with {\arrow{>}}},postaction={decorate}}}
\tikzset{-<<-/.style={decoration={markings, mark=at position 0.55 with {\arrow{<}}},postaction={decorate}}}

\newcommand{\edgein}[7]{\fill (#1) circle (1pt) node [#5] {#2}; \fill (#3) circle (1pt) node [#5] {#4}; \draw[->-=.5] (#1) to node [#7] {#6} (#3);}

\newcommand{\squaresetting}[1]{\coordinate[label=below left:\rlap{$s_{00}$}\phantom{ss}] (s00) at (0,-0.5);
\coordinate[label=above left:\rlap{$s_{01}$}\phantom{ss}] (s01) at (0,0.7);
\coordinate[label=below right:\phantom{ss}\llap{$s_{10}$}] (s10) at (1.2,-0.5);
\coordinate[label=above right:\phantom{ss}\llap{$s_{11}$}] (s11) at (1.2,0.7);
\foreach \point in {s00,s01,s10,s11} \fill [black] (\point) circle (#1);}

\newcommand{\drawsinglesquareone}[4]{\begin{tikzpicture}[baseline=0, font=\tiny]
\coordinate[label={[shift={(-0.05,-0.35)}]:$s_{#1}$}] (s00) at (0,-0.5);
\coordinate[label={[shift={(-0.05,-0.05)}]:$s_{#2}$}] (s01) at (0,0.7);
\coordinate[label={[shift={(0.05,-0.35)}]:$s_{#3}$}] (s10) at (1.2,-0.5);
\coordinate[label={[shift={(0.05,-0.05)}]:$s_{#4}$}] (s11) at (1.2,0.7);
\foreach \point in {s00,s01,s10,s11} \fill [black] (\point) circle (1pt);}
\newcommand{\drawsinglesquaretwo}[8]{\draw[#1,thick] (s00) to node[left=-1pt] {$#2$} (s01);
\draw[#3,thick] (s01) to node[above=-1pt] {$#4$} (s11);
\draw[#5,thick] (s10) to node[right=-1pt] {$#6$} (s11);
\draw[#7,thick] (s00) to node[below=-1pt] {$#8$} (s10);\end{tikzpicture}
}

\let\origmaketitle\maketitle
\def\maketitle{
  \begingroup
  \def\uppercasenonmath##1{} 
  \origmaketitle
  \endgroup
}
\begin{document}
\title[Quaternionic Lattices and a Fake Quadric in Characteristic $2$]{Quaternionic Arithmetic Lattices of Rank $2$\\and a Fake Quadric\\in Characteristic $2$}
\author{Nithi Rungtanapirom}
\address{Nithi Rungtanapirom, Institut f\"ur Mathematik,  Goethe-Universit\"at Frankfurt, Ro\-bert-Mayer-Str. {6-8}, 60325~Frankfurt am Main, Germany}
\email{rungtana@math.uni-frankfurt.de}
\maketitle

\begin{abstract}
We construct a torsion-free arithmetic lattice in $\PGL_2(\bF_2\prr{t})\times\PGL_2(\bF_2\prr{t})$ arising from a quaternion algebra over $\bF_2(z)$. It is the fundamental group of a square complex with universal covering $T_3\times T_3$, a product of trees with constant valency $3$, which has minimal Euler characteristic. Furthermore, our lattice gives rise to a fake quadric over $\bF_2\prr{t}$ by means of non-archimedean uniformization.
\end{abstract}

\tableofcontents


\section{Introduction}
This paper deals with a quaternionic arithmetic lattice which is interesting from several aspects. In geometric group theory, one is interested in torsion-free lattices acting on a product of two trees. The quotient under such a group action is a finite square complex, see for example the square complex $\Sigma$ in Theorem \ref{thmABC:lattice} below. In fact, this square complex $\Sigma$ arose first in experiments performed by Alina Vdovina providing evidence that its fundamental group is an arithmetic group. This rarely happens. From a number theoretic viewpoint, an arithmetic lattice like $\Gamma$ in Theorem \ref{thmABC:lattice} can rarely at the same time be torsion-free and yield a quotient under the Bruhat-Tits action with minimal number of vertices and hence minimal Euler characteristic.

\smallskip

Let $T_n$ be the tree with constant valency $n$. Our main result can be stated as follows:

\begin{thmABC}[see Propositions \ref{prop:presentationofLambda} and \ref{prop:fundgrpinGR1}] \label{thmABC:lattice}
There exists an arithmetic lattice $\Lambda$ arising from a quaternion algebra $[z,1+z^3)$ over $\bF_2(z)$ with the following properties:
\begin{enumerate}
\item Its Bruhat-Tits action on the product $T_3\times T_3$ is simply transitive on the vertices.
\item It is isomorphic to the fundamental group of the quotient orbispace of the square complex $\Sigma$ associated to the $V_4$-structure $(\dA,\dB)$ from Lemma \ref{lem:V4inGR} by $V_4$-action, see also Definition \ref{def:squarecomplex}. This square complex has four vertices, twelve edges and nine squares. The presentation of $\Lambda$ can be given by means of orbispace fundamental group as follows:
\begin{equation} \label{eq:Lambdainintro}
\Lambda = \bigl\langle b_1,b_2,c_1,c_2 ~ \big| ~ c_1^2, ~ c_2^2, ~ c_1c_2 = c_2c_1, ~ b_1b_2c_1b_2, ~ b_1c_2b_1b_2^{-1} \bigr\rangle.
\end{equation}
\item It has a torsion-free normal subgroup $\Gamma$ such that $\Lambda/\Gamma\cong V_4$. This is the fundamental group of the square complex $\Sigma$ from $(2)$.
\end{enumerate}
\end{thmABC}

Our lattice also plays a prominent role in the construction of a fake quadric in characteristic $2$ by means of non-archimedean uniformization as follows:

\begin{thmABC}[see Theorem \ref{thm:fakequadricinchar2}] \label{thmABC:fakequadric}
Let $\Gamma$ be as in Theorem \ref{thmABC:lattice} and $X_\Gamma$ be the scheme over $\bF_2\pbb{t}$ constructed in Proposition \ref{prop:algebraization}. The generic fiber of $X_\Gamma$ is a fake quadric over $\bF_2\prr{t}$.
\end{thmABC}


Several torsion-free quaternionic arithmetic lattices are known that act simply transitively on the vertices of a product of trees. A torsion-free $\{p,\ell\}$-arithmetic lattice for the Hamiltonian quaternions over $\bQ$ which acts simply transitively on the vertices of $T_{p+1}\times T_{\ell+1}$ for two different primes $p,\ell\equiv1\bmod4$ was constructed by Mozes in \cite[\textsection3]{mozes} and Burger and Mozes in \cite[\textsection2.4]{bm00}. Its local permutation structure in the sense of \cite[\textsection1]{bm00} was used in the construction of Ramanujan graphs in \cite{lps-ramanujan}. This result was later generalized to arbitrary odd primes $p\neq\ell$ by Rattaggi in \cite[Ch.3]{rattaggi}. In \cite{stix-vdovina}, Stix and Vdovina constructed a series of torsion-free lattices in quaternion algebras over the rational function field $\bF_q(z)$, where $q$ is an odd prime power, which acts on the vertices of $T_{q+1}\times T_{q+1}$ simply transitively, while its local permutation structure had been used in the construction of Ramanujan graphs in \cite{morgenstern}.

\smallskip

In characteristic $2$, however, it is impossible to find such a torsion-free quaternionic lattice over a function field over $\bF_q$ with simply transitive action on $T_{q+1}\times T_{q+1}$. The reason is that if $N$ is the number of vertices of the quotient square complex $\Sigma$, then the Euler charateristic is
\[
\chi(\Sigma) = \frac14N(q-1)^2,
\]
so that $N$ must be divisible by $4$ and cannot be $1$. The lattice $\Gamma$ constructed in this paper is torsion-free and yields precisely the minimal possible number of vertices, namely four.

\smallskip


We can also determine an explicit presentation for $\Gamma$. In fact, it is known that an $S$-arithmetic subgroup in a reductive group is finitely presented in the number field case and also under certain conditions in the function field case, see \cite{borel-serre}, \cite{serre70}, \cite{serre79} and \cite{behr}. Such a presentation can be determined in general by finding first a Dirichlet fundamental domain, but the explicit results are limited up to now. An algorithm for computing a presentation was given by Voight in \cite{voight-fuchsian} for Fuchsian groups and B\"ockle and Butenuth in \cite{bb-quaterniongraphs} for quaternionic arithmetic lattices of rank $1$ over $\bF_q(t)$. Furthermore, Swan described an algorithm for determining a presentation of $\SL_2(O_K)$, where $K$ is an imaginary quadratic number field, in \cite{swan-sl2-quadimg}.

\smallskip

Unlike the rank $1$ case mentioned above, only few (explicit) results for arithmetic lattices of rank at least $2$ are known so far. In \cite{kirchheimer-wolfart}, Kirchheimer and Wolfart gave an algorithm for computing presentations of the Hilbert modular groups $\PSL_2(O_K)$, where $K$ is a real quadratic number field of class number one. An algorithm for presenting $S$-arithmetic groups of a definite rational quaternion algebra was given by Chinburg, Friedlander, Howe, Kosters, Singh, Stover, Zhang and Ziegler in \cite{quat-S-group}. In positive characteristic, Papikian computed the Euler-Poincar\'{e} characteristic of the simplicial complexes for arithmetic lattices arising from certain central division algebras over $\bF_q(t)$ in \cite{papikian}, while Stix and Vdovina gave explicit presentations for a series of quaternionic arithmetic lattices over $\bF_q(t)$ in \cite{stix-vdovina}.

\smallskip

Contrary to lattices in \cite{mozes}, \cite{bm00}, \cite{rattaggi} and \cite{stix-vdovina}, where the resulting quotients are square complexes with exactly one vertex, our lattice yields a quotient square complex with four vertices. Consequently, it is more complicated to find such an appropriate square complex and to compare the structure of its fundamental group to an arithmetic lattice. Our strategy is to embed our lattice as a normal subgroup of index $4$ in an $S$-arithmetic lattice $\Lambda$ which is not torsion-free but acts on the vertices simply transitively, so that we can find a presentation of $\Lambda$ by means of orbispace fundamental group and hence also $\Gamma$, see Propositions \ref{prop:presentationofLambda} and \ref{prop:fundgrpinGR1}.

\smallskip


The search for our arithmetic lattice was also partly motivated by the construction of a fake quadric by means of non-archimedean uniformization. Based on Mumford's construction in the $1$-dimensional case in \cite{mumford-unifcurves}, this technique was developed independently by Kurihara in \cite{kurihara} and Mustafin in \cite{mustafin}. Later, in \cite{mumford-fakeP2}, Mumford employed this technique to construct a fake projective plane, a minimal surface of general type with $\rc_1^2=9$, $\rc_2=3$ and trivial Albanese variety, i.e.~these numerical invariants coincide with those of the projective plane. Also Stix and Vdovina employed this technique in \cite{stix-vdovina} to construct a class of algebraic surfaces of general type with Chern ratio $\rc_1^2/\rc_2=2$, trivial Albanese variety and non-reduced Picard variety for any odd prime power $q$, which yields a fake quadric for $q=3$.

\smallskip

The notion of fake quadrics is motivated by the notion of fake projective planes. A fake quadric is a minimal surface of general type such that $\rc_1^2=8$, $\rc_2=4$ and the Albanese variety is trivial, i.e.~these numerical invariants coincide with those of a quadric surface. In characteristic $0$, all fake quadrics known so far have complex analytic uniformization by $\bH\times\bH$, where $\bH$ is the upper half plane. Hence they are of the form $X_\Gamma:=\Gamma\backslash(\bH\times\bH)$ for some lattice $\Gamma$ in the holomorphic isometry group $\Isom^h(\bH\times\bH)$. These can be divided into two classes:

\begin{enumerate}
\item
\emph{Reducible fake quadrics,} i.e.~those $X_\Gamma$'s such that $\Gamma=\pi_1(X_\Gamma)$ is reducible, i.e.~ commensurable with a product of two lattices in $\PSL_2(\bR)\cong\Isom^h(\bH)$. In this case $X$ is the quotient of the product of two curves of genus $\geq2$ by a finite group. The first such surface was given by Beauville in \cite{beauville-cpxsurf}, exercise X.4. In \cite{bauer-catanese-grunewald}, Bauer, Catanese and Grunewald identified 17 families of such surfaces. Later, in \cite{frapporti}, Frapporti found one further family and hence completed the classification of reducible fake quadrics.

\item
\emph{Irreducible fake quadrics,} i.e.~those $X_\Gamma$'s that are not reducible. Here the lattice $\Gamma$ is arithmetic by \cite[Ch.IX Thm.1.11]{margulis} and arises from a quaternion algebra over a totally real number field. Such examples in the stable case, i.e.~$\Gamma$ is contained in $\PSL_2(\bR)\times\PSL_2(\bR)\leq\Isom^h(\bH\times\bH)$, are given by Kuga and Shavel in \cite{shavel} as well as by D\u{z}ambi\'{c} in \cite{amir-dzambic-fakequadrics}, where the case of quaternion algebras over real quadratic fields is established in details. Later, in \cite{linowitz-stover-voight}, Linowitz, Stover and Voight established the general case and listed maximal arithmetic lattices in $\Isom^h(\bH\times\bH)$ that can contain those $\Gamma$ such that $X_\Gamma$ is a fake quadric.
\end{enumerate}

In positive characteristic, only few results are known so far. The construction of the known ones may be sketched as follows: Let $R:=\bF_q\pbb{t}$, $K:=\Quot(R)$, $\Gamma$ be a torsion-free lattice in $\PGL_2(K)\times\PGL_2(K)$ and $\cY/\Spf(R)$ be the formal scheme from Definition \ref{def:formalwonderful} (hence its generic fibre is the Drinfeld upper half plane $\Omega_K^1$). Then the quotient $\Gamma\backslash(\cY\times\cY)$ has ample line bundle defined by relative log-differentials and can thus be algebraized to a projective scheme $X_\Gamma$ over $R$ by Grothendieck's formal GAGA principle. If the quotient square complex of $\Gamma$ under Bruhat-Tits action has $N$ vertices, the generic fiber is a minimal surface with Euler characteristic
\[
\chi(X_{\Gamma,K}) = \frac14N(q-1)^2.
\]
It follows that $\chi(X_{\Gamma,K})=1$ if and only if $N=1$, $q=3$ or $N=4$, $q=2$. In particular, this construction can yield a fake quadric only for $q=2,3$. The case $q=3$ was studied by Stix and Vdovina in \cite{stix-vdovina} as mentioned before. The case $q=2$ leads to the question of finding a torsion-free lattice which yields a quotient complex with four vertices as discussed in this paper.

\subsection*{Construction of $\Gamma$ from Theorem \ref{thmABC:lattice} (sketch)} We set $K:=\bF_2(z)$ and define the following algebra with non-commuting variables $I,J$ over $K$:
\[
Q := \left[\frac{z,1+z^3}{K}\right) = K\{I,J\}\big/\{I^2+I=z, ~ J^2=1+z^3, ~ JI=(I+1)J\}.
\]
Furthermore, consider the following rings:
\[
R_0:=\bF_2[z,\tfrac{1}{z}], \quad R_1:=\bF_2[z,\tfrac{1}{z(1+z)}] \quad \text{and} \quad R:=\bF_2[z,\tfrac{1}{z(1+z^3)}].
\]
The basis $1,I,J,IJ$ defines a maximal order $\fO_0$ over $R_0$. The algebraic group $G:= \PGL_{1,\fO}$ over $R_0$ can be defined as follows: Let $x_0,x_1,x_2,x_3$ be the coordinates for $\bP_{R_0}^3$. The open subscheme $G\subseteq\bP_{R_0}^3$ is defined as the complement of the closed subscheme given by the polynomial
\[
\rn(x_0+x_1I+x_2J+x_3IJ) \in R_0[x_0,x_1,x_2,x_3].
\]
The group law on $G$ is given by the multiplication of quaternions by identifying each point $[x_0:x_1:x_2:x_3]\in\bP_{R_0}^3$ with the quaternion $x_0+x_1I+x_2J+x_3IJ$ up to scalar multiple.

\smallskip

To see that $G(R)$ has a subgroup $\Lambda$ with the presentation as in \eqref{eq:Lambdainintro}, observe that the generators have the following images under the embeddings $\rho_y$ in $\PGL_2(\bF_2\prr{y})$ from Lemma \ref{splitting1} and $\rho_t$ in $\PGL_2(\bF_2\prr{t})$ from Lemma \ref{splitting2}:

\begin{center}
\begin{tabular}{c|cc}
& $\rho_y$ & $\rho_t$ \\[0.5ex] \hline
&& \\[-2.5ex] 
$b_1$ & $\left(\begin{smallmatrix}(1+z)y&1+z^3\\1&(1+z)(1+y)\end{smallmatrix}\right)$ & $\left(\begin{smallmatrix}(1+u)(1+u+t)&u+u^4\\1&(1+u)(u+t)\end{smallmatrix}\right)$ \\[1ex]
$b_2$ & $\left(\begin{smallmatrix}z+z^2+(1+z)y&(1+z^3)(1+y)\\y&1+z^2+(1+z)y\end{smallmatrix}\right)$ & $\left(\begin{smallmatrix}(1+u)t&(1+t)(1+u^3)\\t/u&(1+u)(1+t)\end{smallmatrix}\right)$ \\[1ex]
$c_1$ & $\left(\begin{smallmatrix}1+z^2&(1+z^3)y\\1+y&1+z^2\end{smallmatrix}\right)$ & $\left(\begin{smallmatrix}u+u^2&(1+u^3)(1+u+t)\\(u+t)/u&u+u^2\end{smallmatrix}\right)$ \\[1ex]
$c_2$ & $\left(\begin{smallmatrix}z+z^2&(1+z^3)y\\1+y&z+z^2\end{smallmatrix}\right)$ & $\left(\begin{smallmatrix}1+u^2&(1+u^3)(1+u+t)\\(u+t)/u&1+u^2\end{smallmatrix}\right)$ \\[1ex]
\end{tabular}
\end{center}
Here $u=z^{-1}$ and $y,t$ are such that $y^2+y=z$ and $t^2+t=u$. Now let $w$ denote the standard vertex of the product of the Bruhat-Tits trees for $\PGL_2(\bF_2\prr{y})$ and $\PGL_2(\bF_2\prr{t})$. Then $w$ is sent by $b_1,b_1^{-1},c_1$ to its vertical and by $b_2,b_2^{-1},c_2$ to its horizontal neighbors under the Bruhat-Tits action. This leads to a generalization of a VH-structure in the sense of \cite[\textsection2.2, Def.4]{stix-vdovina}. We call this a \textbf{$V_4$-equivariant vertical-horizontal structure}, in short \textbf{$V_4$-structure}.

\smallskip

In analogy to \cite[\textsection2.2.1]{stix-vdovina}, a $V_4$-structure gives rise to a square complex $\Sigma$ with four vertices and a $V_4$-action. Consequently, $\Lambda$ is isomorphic to the orbispace fundamental group $\pi_1^\orb(\Sigma,V_4)$, so that $\Lambda$ has a presentation as in \eqref{eq:Lambdainintro}. Furthermore, we can show that $G(R_1)=\Lambda$ and the maximal arithmetic lattice $G(R)$ has the following presentation
\[
G(R) = \left\langle b_1,b_2,c_1,c_2,d ~ \left| \begin{array}{c}c_1^2, ~ c_2^2, ~ d^2, ~ c_1c_2 = c_2c_1, ~ c_1d=dc_1, ~ c_2d=dc_2,\\
b_1b_2c_1b_2, ~ b_1c_2b_1b_2^{-1}, ~ db_1db_1, ~ db_2db_2 \end{array} \right. \right\rangle,
\]
see Theorem \ref{presentationGR}. The subgroup
\[
\Gamma := \ker(\Lambda\cong\pi_1^\orb(\Sigma;V_4)\longtwoheadrightarrow V_4)
\]
is then a normal subgroup of index $4$ in $\Lambda$ corresponding to the fundamental group $\pi_1(\Sigma)$. Hence the quotient $\Gamma\backslash(T_3\times T_3)$ is a square complex with four vertices as desired.

\smallskip

Based on our lattice $\Gamma$, we can construct an algebraic surface of general type over $\bF_2\prr{t}$ by means of non-archimedean uniformization. Since the resulting quotient complex has four vertices, this surface has the Chern numbers $\rc_1^2=8$ and $\rc_2=4$. Its Albanese variety can be computed via Kummer \'etale cohomology using the local permutation groups in the sense of \textsection\ref{subsec:localpermgroup}, which is a special case of \cite[\textsection1]{bm00} and a modification of \cite[\textsection5.1]{stix-vdovina}. In the end, the constructed surface is indeed a fake quadric over $\bF_2\prr{t}$ as will be proved in Theorem \ref{thm:fakequadricinchar2}.

\subsection*{Acknowledgements}
I would like to thank Jakob Stix for numerous discussions while I was writing my PhD thesis, from which this article is originated, and for his careful reading of this paper. My thanks also go to John Voight for several suggestions to an earlier version.

\subsection*{Outline}
In Section \ref{sec:sqcplxV4}, we introduce the notion of a $V_4$-structure $(\dA,\dB)$ of a group $\Lambda$ and associate to it a square complex $\Sigma_{\dA,\dB}$ with $V_4$-action as well as local permutation groups $P_j^\dA,P_i^\dB$. The orbispace fundamental group $\pi_1^\orb(\Sigma_{\dA,\dB},V_4)$ has a nice presentation as shown in Proposition \ref{prop:pi1orbSAB} and can be compared under certain conditions to the group $\Lambda$ as shown in Theorem \ref{thm:compareonTxT}.

Section \ref{sec:quatlattices} is devoted to arithmetic quaternion lattices. Here we start with a quick review on the arithmetic theory of quaternion algebras in characteristic $2$, then introduce arithmetic lattices in the quaternion algebra of primary interest and determine their presentation by means of Bruhat-Tits action. This is done by determining a stabilizer and finding a subgroup with a $V_4$-structure. We also compute the local permutation groups of this $V_4$-structure.

The construction of a fake quadric in characteristic $2$ by means of non-archimedean uniformization is discussed in Section \ref{sec:fakequadric}. This is based on the torsion-free lattice from Proposition \ref{prop:fundgrpinGR1}. Here the local permutation groups of the $V_4$-structure from Lemma \ref{lem:V4inGR} play an important role in computing the Albanese variety. This proves that the constructed surface is indeed a fake quadric.

\subsection*{Notation and Terminology}
By a \textbf{lattice} in a locally compact group $G$, we mean a discrete subgroup $\Gamma$ such that the quotient $\Gamma\backslash G$ has finite volume with respect to the Haar measure induced by $G$.

\smallskip

An algebraic closure of a field $K$ will be denoted by $\bar{K}$. If $K$ is a global field and $S$ is a set of places in $K$, we denote the ring of $S$-integers by $O_{K,S}$. The normalized valuation on $K$ with respect to the place $\fp$ is denoted by $\nu_\fp$

\smallskip

In the special case $K=\bF_2(z)$, we denote the completion of $K$ at the place given by a closed point $x\in\bP^1_{\bF_2}$ by $K_x$ and its normalized valuation by $\nu_x$. The Bruhat-Tits tree for $\PGL_2(K_x)$ will be denoted by $T_x$ and its standard vertex by $w_x$. Note that for $n\in\bN$, we also denote the infinite tree with constant valency $n$ by $T_n$, but it should be clear from the context whether the index ($x$ or $n$) stands for a point in $\bP_{\bF_2}^1$ or a natural number.

\smallskip

For the explicit computation with the quaternion algebra $Q=\bigl[\frac{z,1+z^3}{\bF_2(z)}\bigr)$, we use the following notation:

\begin{center}
\begin{tabular}{p{1.75cm}p{12.5cm}}
$K$ & the global function field $\bF_2(z)$ of the projective curve $\bP^1_{\bF_2}$,\\
$I,J$ & generators of $Q$ with $I^2=I+z$, $J^2=1+z^3$ and $JI=(I+1)J$,\\
$G$ & the projective linear group $\PGL_{1,Q}=\GL_{1,Q}/\bG_m$ over $K$,\\
$R_0,R_1,R$ ~ & the rings $\bF_2[z,\frac1z]$, $\bF_2[z,\frac{1}{z+z^2}]$ and $\bF_2[z,\frac{1}{z+z^4}]$ respectively,\\
$\zeta\in\bar{\bF}_2$ & a primitive third root of unity,\\
$\rho_y,\rho_t$ & the embeddings from Lemma \ref{splitting1} and \ref{splitting2} respectively,\\
$\gamma_v,\gamma_h$ & commuting generators of $V_4$ of order $2$,\\
$\dA,\dB$ & subsets of $G(R_1)$ consisting of $b_1,b_1^{-1},c_1$ resp.~$b_2,b_2^{-1},c_2$, compare Def.\,\ref{introduceelements},\\
$d$ & explicit element in $G(R)$ from Proposition \ref{prop:stabilizer}.
\end{tabular}
\end{center}

\smallskip

A \textbf{square complex} is a two-dimensional combinatorial cell complex such that each $2$-cell is attached to a path of length $4$ beginning and ending at the same point. If $\Sigma$ is a square complex, the set of its vertices, (oriented) edges and squares will be denoted by $\bV(\Sigma)$, $\bE(\Sigma)$ and $\bS(\Sigma)$ respectively. The reverse edge of $\alpha\in\bE(\Sigma)$ will be denoted by $\bar\alpha$, the set of unoriented edges of $\Sigma$ by $[\bE(\Sigma)] = \bE(\Sigma)/(\alpha\sim\bar\alpha)$ and the unoriented edge associated to $\alpha\in\bE(\Sigma)$ by $[\alpha]$.

\smallskip

A square complex $\Sigma$ is said to have a \textbf{vertical-horizontal structure}, in short \textbf{VH-structure}, if there is a partition $\bE(\Sigma) = \bE(\Sigma)_v\sqcup\bE(\Sigma)_h$ of the edges into the vertical and horizontal ones such that each path attached to a square alternates between vertical and horizontal edges. In this case, the set of unoriented vertical edges and unoriented horizontal edges will be denoted by $[\bE(\Sigma)]_v$ and $[\bE(\Sigma)]_h$ respectively.

\smallskip

Furthermore, for a given $V_4$-structure $(\dA,\dB)$, we use the following notation:

\begin{center}
\begin{tabular}{p{1.75cm}p{12.5cm}}
$\Sigma_{\dA,\dB}$ & the square complex with four vertices associated to $(\dA,\dB)$,\\
$\Sym(X)$ & the permutation group of the set $X$,\\
$G\wr\Sym(I)$ & the wreath product $G^I\rtimes\Sym(I)$,\\
$\widehat{\cdot}$ & the transposition between $0$ and $1$,\\
$t^j_{(a,i)},t^i_{(b,j)}$ & the bijection between the squares attached to a vertical/horizontal edge and\\[-0.5ex]
& the horizontal/vertical edges attached to a vertex, see Definition \ref{def:weirdbij},\\
$P^\dA_j,P^\dB_i$ & the local permutation groups associated to $(\dA,\dB)$,\\[0.25ex]
$\tau^\dA,\tau^\dB$ & the maps sending each $a\in\dA$ resp. $b\in\dB$ to its inverse.
\end{tabular}
\end{center}

\section{Square complexes with $V_4$-actions} \label{sec:sqcplxV4}
\subsection{Square complex arising from a $V_4$-structure}
We begin with the definition of a $V_4$-equivariant vertical-horizontal structure and associate to it a square complex with a $V_4$-action.

\begin{defi}\label{def:V4structure}
A \textbf{$V_4$-equivariant vertical-horizontal structure}, in short \textbf{$V_4$-structure}, of a group $\Lambda$ is an ordered pair $(\dA,\dB)$ of finite subsets $\dA,\dB\subseteq \Lambda$ satisfying the following properties:
\begin{enumerate}
\item
$\dA\cup\dB$ generates $\Lambda$.
\item
$\dA$ and $\dB$ are closed under taking inverse.
\item
$\dA\dB=\dB\dA$ and the maps
$$\dA\times\dB \longto \dA\dB, ~ (a,b)\longmapsto ab ~\, \text{resp.} ~ (a,b)\longmapsto ba$$
are well-defined bijections.
\end{enumerate}
Note that in contrast to a VH-structure in \cite[\textsection2]{stix-vdovina}, we also allow $2$-torsions in $\dA,\dB$ and $\dA\dB$. A $V_4$-structure $(\dA,\dB)$ is said to be \textbf{inverse-stable} if for $a,a'\in\dA$ and $b,b'\in\dB$ such that $ab'=ba'$, we have $a^{-1}b'^{-1} = b^{-1}a'^{-1}$.
\end{defi}

\begin{defi} \label{def:squarecomplex}
The \textbf{square complex associated to a $V_4$-structure} $(\dA,\dB)$ of a group $\Lambda$, denoted by $\Sigma_{\dA,\dB}$, is a square complex with the following data: The vertex set is given by $\bV(\Sigma_{\dA,\dB}):=\{s_{00},s_{01},s_{10},s_{11}\}$. The set of oriented edges is given by
\[
\bE(\Sigma_{\dA,\dB}):=\bE(\Sigma_{\dA,\dB})_v \sqcup \bE(\Sigma_{\dA,\dB})_h,
\]
where $\bE(\Sigma_{\dA,\dB})_v$ resp.~$\bE(\Sigma_{\dA,\dB})_h$ are the sets of vertical resp.~horizontal edges defined by
\[
\bE(\Sigma_{\dA,\dB})_v := \bigl\{ (a,i),\overline{(a,i)} \, \big| \, a\in\dA,i\in I \bigr\}
\quad \text{and} \quad \bE(\Sigma_{\dA,\dB})_h := \bigl\{ (b,j),\overline{(b,j)} \, \big| \, b\in\dA,j\in I \bigr\}.
\]
Here each $(a,i)\in\dA\times\{0,1\}$ represents the vertical edge from $s_{i0}$ to $s_{i1}$ labeled by $a$ and each $(b,j)\in\dB\times\{0,1\}$ the horizontal edge from $s_{0j}$ to $s_{1j}$ labeled by $b$. Finally, the set of squares in $\Sigma_{\dA,\dB}$ is given by
\[
\bS(\Sigma_{\dA,\dB}) := \bigl\{ \left.[a,b';b,a'] \, \right| \, a,a'\in\dA, ~ b,b'\in\dB ~ \text{such that} ~ ab'=ba' \bigr\}.
\]
Here each square $[a,b';b,a']$ is glued to the path $\bigl((a,0),(b',1),\overline{(a',1)},\overline{(b,0)}\bigr)$, or equivalently to $\bigl((b,0),(a',1),\overline{(b',1)},\overline{(a,0)}\bigr)$, and can be hence visualized by the following picture.
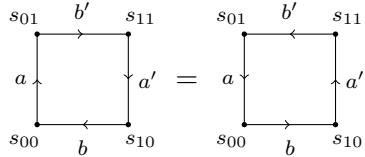
\[
\begin{tikzpicture}[font=\scriptsize,baseline=0]
\squaresetting{1pt}
\draw[->-=.5] (s00) to node [left] {$a$} (s01);
\draw[->-=.5] (s01) to node [above=2pt] {$b'$} (s11);
\draw[->-=.5] (s11) to node [right] {$a'$} (s10);
\draw[->-=.5] (s10) to node [below=2pt] {$b$} (s00);
\end{tikzpicture}
=
\begin{tikzpicture}[font=\scriptsize,baseline=0]
\squaresetting{1pt}
\draw[->-=.5] (s01) to node [left] {$a$} (s00);
\draw[->-=.5] (s11) to node [above=2pt] {$b'$} (s01);
\draw[->-=.5] (s10) to node [right] {$a'$} (s11);
\draw[->-=.5] (s00) to node [below=2pt] {$b$} (s10);
\end{tikzpicture}
\]
Note that $\#\bS(\Sigma_{\dA,\dB}) = (\#\dA)(\#\dB)$ by the definition of a $V_4$-structure. Furthermore, $(\dA,\dB)$ is inverse-stable if and only if there is an action of the group $\langle \delta \rangle \cong \bZ/2\bZ$ fixing the vertices and interchanging the edges $(a,i)\in\dA\times I$ with $(a^{-1},i)$ and $(b,j)\in\dB\times I$ with $(b^{-1},j)$, i.e.
\[
[a,b';b,a'] \in \bS(\Sigma_{\dA,\dB}) \quad \text{if and only if} \quad [a^{-1},b'^{-1};b^{-1},a'^{-1}] \in \bS(\Sigma_{\dA,\dB}).
\]
\end{defi}

The name $V_4$-equivariant vertical-horizontal structure comes from the fact that we can define a $V_4$-action on the square complex $\Sigma_{\dA,\dB}$. We shall use the following notation for the group $V_4$:
\[
V_4 = \left\langle \gamma_v,\gamma_h \mid \gamma_v^2=\gamma_h^2=1, ~ \gamma_v\gamma_h = \gamma_h\gamma_v \right\rangle \cong (\bZ/2\bZ)\times(\bZ/2\bZ),
\]

\begin{defi} \label{def:V4action}
We define the action of the group $V_4$ on the square complex $\Sigma=\Sigma_{\dA,\dB}$ as follows: On the set $\bV(\Sigma)$ of vertices, $\gamma_v$ interchanges $s_{ij}$ with $s_{i(1-j)}$, while $\gamma_h$ interchanges $s_{ij}$ with $s_{(1-i)j}$ ($i,j\in\{0,1\}$). This can be visualized by the diagram below.

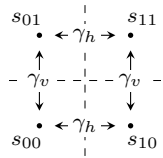
\begin{center}
\begin{tikzpicture}[font=\scriptsize]
\squaresetting{1pt} \path[dashed] (-0.4,0.1) edge [very thin] (1.6,0.1) (0.6,1.1) edge [very thin] (0.6,-0.9);
\path[>=stealth,<->] ($(s00)+(0,0.2)$) edge node [style={fill=white,inner sep=1.5pt}] {$\gamma_v$} ($(s01)-(0,0.2)$) ($(s10)+(0,0.2)$) edge node [style={fill=white,inner sep=1.5pt}] {$\gamma_v$} ($(s11)-(0,0.2)$) ($(s00)+(0.2,0)$) edge node [style={fill=white,inner sep=1.5pt}] {$\gamma_h$} ($(s10)-(0.2,0)$) ($(s01)+(0.2,0)$) edge node [style={fill=white,inner sep=1.5pt}] {$\gamma_h$} ($(s11)-(0.2,0)$);
\end{tikzpicture}
\end{center}
To define the action of $V_4$ on $\bE(\Sigma)$, we let $\gamma_v$ interchange the edge $(a,i)\in\dA\times\{0,1\}$ with $\overline{(a^{-1},i)}$, and the edge $(b,0)$ with $(b,1)$ for $b\in\dB$. Furthermore, we let $\gamma_h$ interchange the edge $(a,0)$ with $(a,1)$ for $a\in\dA$, and the edge $(b,j)\in\dB\times\{0,1\}$ with $\overline{(b^{-1},j)}$.

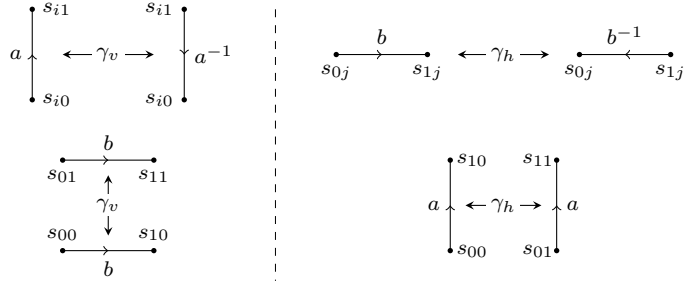
\begin{center}
\begin{tikzpicture}[font=\scriptsize]
\edgein{0,0.4}{$s_{i0}$}{0,1.6}{$s_{i1}$}{right}{$a$}{left} \edgein{2,1.6}{$s_{i1}$}{2,0.4}{$s_{i0}$}{left}{$a^{-1}$}{right}
\edgein{0.4,-0.4}{$s_{01}$}{1.6,-0.4}{$s_{11}$}{below}{$b$}{above} \edgein{0.4,-1.6}{$s_{00}$}{1.6,-1.6}{$s_{10}$}{above}{$b$}{below}
\edgein{4,1}{$s_{0j}$}{5.2,1}{$s_{1j}$}{below}{$b$}{above} \edgein{8.4,1}{$s_{1j}$}{7.2,1}{$s_{0j}$}{below}{$b^{-1}$}{above}
\edgein{5.5,-1.6}{\!$s_{00}$}{5.5,-0.4}{\!$s_{10}$}{right}{$a$}{left} \edgein{6.9,-1.6}{$s_{01}$\!\!}{6.9,-0.4}{$s_{11}$\!\!}{left}{$a$}{right}
\path[>=stealth,<->] (0.4,1) edge node [style={fill=white,inner sep=1.6pt}] {$\gamma_v$} (1.6,1) (1,-0.6) edge node [style={fill=white,inner sep=1.6pt}] {$\gamma_v$} (1,-1.4);
\path[>=stealth,<->] (5.6,1) edge node [style={fill=white,inner sep=1.6pt}] {$\gamma_h$} (6.8,1) (5.7,-1) edge node [style={fill=white,inner sep=1.6pt}] {$\gamma_h$} (6.7,-1);
\path[dashed] (3.2,1.6) edge [very thin] (3.2,-2);
\end{tikzpicture}
\end{center}
The action of $V_4$ on the vertices and edges of $\Sigma$ forces the action on the square of $\Sigma$ to be as shown in the diagram below.

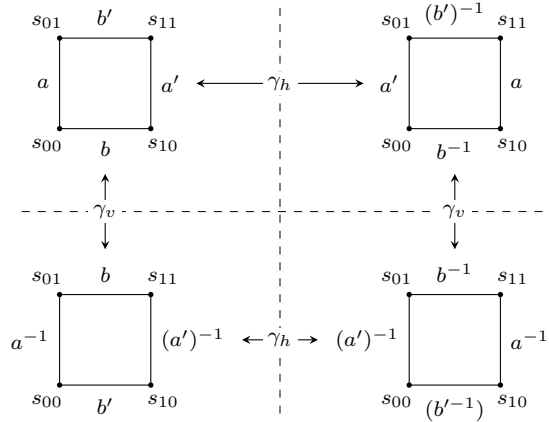
\begin{center}
\begin{tikzpicture}[font=\scriptsize,baseline=0]
\coordinate[label=below left:\rlap{$s_{00}$}\phantom{ss}] (s00a) at (0,1.1);
\coordinate[label=above left:\rlap{$s_{01}$}\phantom{ss}] (s01a) at (0,2.3);
\coordinate[label=below right:\phantom{ss}\llap{$s_{10}$}] (s10a) at (1.2,1.1);
\coordinate[label=above right:\phantom{ss}\llap{$s_{11}$}] (s11a) at (1.2,2.3);
\foreach \point in {s00a,s01a,s10a,s11a} \fill [black] (\point) circle (1pt);
\path [-] (s00a) edge node [left] {$a$} (s01a) edge node [below=1pt] {$b$} (s10a) (s11a) edge node [above=1pt] {$b'$} (s01a) edge node [right] {$a'$} (s10a);
\coordinate[label=below left:\rlap{$s_{00}$}\phantom{ss}] (s00b) at (4.6,1.1);
\coordinate[label=above left:\rlap{$s_{01}$}\phantom{ss}] (s01b) at (4.6,2.3);
\coordinate[label=below right:\phantom{ss}\llap{$s_{10}$}] (s10b) at (5.8,1.1);
\coordinate[label=above right:\phantom{ss}\llap{$s_{11}$}] (s11b) at (5.8,2.3);
\foreach \point in {s00b,s01b,s10b,s11b} \fill [black] (\point) circle (1pt);
\path [-] (s00b) edge node [left] {$a'$} (s01b) edge node [below=1pt] {$b^{-1}$} (s10b) (s11b) edge node [above=1pt] {$(b')^{-1}$} (s01b) edge node [right] {$a$} (s10b);
\coordinate[label=below left:\rlap{$s_{00}$}\phantom{ss}] (s00c) at (0,-2.3);
\coordinate[label=above left:\rlap{$s_{01}$}\phantom{ss}] (s01c) at (0,-1.1);
\coordinate[label=below right:\phantom{ss}\llap{$s_{10}$}] (s10c) at (1.2,-2.3);
\coordinate[label=above right:\phantom{ss}\llap{$s_{11}$}] (s11c) at (1.2,-1.1);
\foreach \point in {s00c,s01c,s10c,s11c} \fill [black] (\point) circle (1pt);
\path [-] (s00c) edge node [left] {$a^{-1}$} (s01c) edge node [below=1pt] {$b'$} (s10c) (s11c) edge node [above=1pt] {$b$} (s01c) edge node [right] {$(a')^{-1}$} (s10c);
\coordinate[label=below left:\rlap{$s_{00}$}\phantom{ss}] (s00d) at (4.6,-2.3);
\coordinate[label=above left:\rlap{$s_{01}$}\phantom{ss}] (s01d) at (4.6,-1.1);
\coordinate[label=below right:\phantom{ss}\llap{$s_{10}$}] (s10d) at (5.8,-2.3);
\coordinate[label=above right:\phantom{ss}\llap{$s_{11}$}] (s11d) at (5.8,-1.1);
\foreach \point in {s00d,s01d,s10d,s11d} \fill [black] (\point) circle (1pt);
\path[dashed] (2.9,2.7) edge [very thin] (2.9,-2.7) (-0.5,0) edge [very thin] (6.5,0);
\path [-] (s00d) edge node [left] {$(a')^{-1}$} (s01d) edge node [below=1pt] {$(b'^{-1})$} (s10d) (s11d) edge node [above=1pt] {$b^{-1}$} (s01d) edge node [right] {$a^{-1}$} (s10d);
\path[>=stealth,<->] (0.6,0.5) edge node [style={fill=white,inner sep=1.8pt}] {$\gamma_v$} (0.6,-0.5) (5.2,0.5) edge node [style={fill=white,inner sep=1.8pt}] {$\gamma_v$} (5.2,-0.5) (1.8,1.7) edge node [style={fill=white,inner sep=1.8pt}] {$\gamma_h$} (4.0,1.7) (2.4,-1.7) edge node [style={fill=white,inner sep=1.8pt}] {$\gamma_h$} (3.4,-1.7);
\end{tikzpicture}
\end{center}
Note that this $V_4$-action is well-defined since the relation $ab'=ba'$ is equivalent to the relations $a'(b')^{-1}=b^{-1}a$, $a^{-1}b=b'(a')^{-1}$ and $(a')^{-1}b^{-1} = (b')^{-1}a^{-1}$. Furthermore, it respects the vertical-horizontal structure of the square complex.
\end{defi}

\subsection{Group theory of $V_4$-structures} \label{subsec:localpermgroup}
Based on \cite[\textsection1]{bm00}, we define here the local permutation groups associated to a $V_4$-structure and begin by fixing the notation of the following special case of wreath product.

\begin{defi}
Let $G$ be a group and $I$ be a set. The \textbf{wreath product} $G\wr\Sym(I)$ is defined as the semidirect product $G^I\rtimes\Sym(I)$ with the action of $\Sym(I)$ on $G^I$ defined by
\[
\sigma ((g_i)_{i\in I}) := (g_{\sigma^{-1}(i)})_{i\in I} \qquad \text{for} \quad (g_i)_{i\in I}\in G^I, ~ \sigma\in\Sym(I).
\]
\end{defi}

\begin{rmk}
If $G=\Sym(X)$ is the permutation group of another set $X$, then the wreath product $\Sym(X)\wr\Sym(I)$ can be embedded in $\Sym(X\times I)$ by defining
\[
\sigma(x,i) = (\sigma_{\bar\sigma(i)}(x),\bar\sigma(i))
\]
for each $(x,i)\in X\times I$ and $\sigma = ((\sigma_i)_i,\bar\sigma)\in\Sym(X)\wr\Sym(I)$.
\end{rmk}

From now on, let $(\dA,\dB)$ be a $V_4$-structure of a group $\Lambda$. To define the local permutation groups associated to the $V_4$-structure $(\dA,\dB)$, we begin with the following notation:

\begin{nota}
Set $I:=\{0,1\}$, \thickspace $\widehat{\cdot}:=(0~1)\in\Sym(I)$, $\dA_i:=\dA\times\{i\}$ and $\dB_j:=\dB\times\{j\}$.
\end{nota}

\begin{defi}\label{def:weirdbij}
For $a\in\dA$, $b\in\dB$ and $i,j\in I$, let $\bS_{(a,i)},\bS_{(b,j)}\subseteq \bS(\Sigma_{\dA,\dB})$ be the set of the squares attached to the edge $(a,i)$ and $(b,j)$ respectively. Furthermore, the mappings
\[
t_{(a,i)}^j : \bS_{(a,i)}\longto\dB_j \quad \text{and} \quad t_{(b,j)}^i : \bS_{(b,j)}\longto\dA_i
\]
are defined by sending each square in $\bS_{(a,i)}$ and $\bS_{(b,j)}$ to its horizontal and vertical edge attached to the vertex $s_{ij}$ respectively.
\end{defi}

To see that $t_{(a,i)}^j$ and $t_{(b,j)}^i$ are bijective for all $a\in\dA$, $b\in\dB$ and $i,j\in I$, we establish first the link at each vertex of $\Sigma_{\dA,\dB}$. Recall that for a square complex $\Sigma$ and $s\in\bV(\Sigma)$, the \textbf{link} $\mathrm{Lk}_s=\mathrm{Lk}_s(\Sigma)$ is the (undirected multi-)graph whose set of vertices is given by the end of the edges in $\Sigma$ attached to $s$, and whose set of edges joining two vertices $a,b\in\mathrm{Lk}_s$ is given by the square corners attached to $a$ and $b$.

\begin{lem} \label{lem:links}
For each $s\in \bV(\Sigma_{\dA,\dB})$, the link $\mathrm{Lk}_{s}$ is a complete bipartite graph with vertical vertices labeled by $\dA$ and horizontal vertices labeled by $\dB$.
\end{lem}

\begin{proof}
We prove this first for $s=s_{01}$. Since the squares of $\Sigma_{\dA,\dB}$ are of the form $[a,b\,;b',a']$, we have to show that for each $a\in\dA$ and $b\in\dB$, there are unique $a'\in\dA$ and $b'\in\dB$ such that $ab=b'a'$. But this is contained in the definition of a $V_4$-structure. Thus the link of $s_{01}$ is a complete bipartite graph. The claim for the other vertices follows since the action of each element in $V_4$ on $\Sigma_{\dA,\dB}$ yields an isomorphism between the links of the corresponding vertices.
\end{proof}

\begin{cor} \label{cor:weirdmapsarebij}
For all $a\in\dA$, $b\in\dB$ and $i,j\in I$, the maps $t_{(a,i)}^j$ and $t_{(b,j)}^i$ are bijective.
\end{cor}

\begin{proof}
The edges of $\mathrm{Lk}_{s_{ij}}$ attached to $a$ are in a one-to-one correspondence to the squares in $\bS_{(a,i)}$ and the edges attached to $b$ to the squares in $\bS_{(b,j)}$. Hence the claim follows since $\mathrm{Lk}_{s_{ij}}$ is a complete bipartite graph.
\end{proof}

From Corollary \ref{cor:weirdmapsarebij}, the following maps are bijective for all $a\in\dA$, $b\in\dB$ and $i,j\in I$:
\[
t_{(a,i)}^j\circ(t_{(a,i)}^{\widehat{j}})^{-1} : \dB_{\widehat{j}}\longto\dB_j \quad \text{and} \quad t_{(b,j)}^i\circ(t_{(b,j)}^{\widehat{i}})^{-1}:\dA_{\widehat{i}}\longto\dA_i.
\]
Hence $t_{(a,i)}^0\circ(t_{(a,i)}^1)^{-1} \sqcup t_{(a,i)}^1\circ(t_{(a,i)}^0)^{-1}$ is a bijection from $\dB_0\sqcup\dB_1=\dB\times I$ to $\dB_1\sqcup\dB_0=\dB\times I$. We can do the same things replacing $(a,i)$ by $(b,j)$ and $\dB$ by $\dA$ and obtain the following definitions.

\begin{defi}
We define the following mappings
\[
\begin{array}{rl}
\dA\times I \longto \Sym(\dB\times I), & (a,i) \longmapsto \sigma^\dB_{(a,i)}:=t_{(a,i)}^0\circ(t_{(a,i)}^1)^{-1} \sqcup t_{(a,i)}^1\circ(t_{(a,i)}^0)^{-1} \quad \text{and}\\
\dB\times I \longto \Sym(\dA\times I), & (b,j) \longmapsto \sigma^\dA_{(b,j)}:=t_{(b,j)}^0\circ(t_{(b,j)}^1)^{-1} \sqcup t_{(b,j)}^1\circ(t_{(b,j)}^0)^{-1}.
\end{array}
\]
Note that these are well-defined since $t_{(a,i)}^j$ and $t_{(b,j)}^i$ are all bijective. The \textbf{local permutation groups associated to the $V_4$-structure $(\dA,\dB)$} are defined by
\begin{align*}
P^\dA_j &:= \bigl\langle \sigma^\dA_{(b,j)} \mid b\in\dB \bigr\rangle \leq \Sym(\dA\times I) \quad \text{and} \\[-0.5ex]
P^\dB_i &:= \bigl\langle \sigma^\dB_{(a,i)} \mid a\in\dA \bigr\rangle \leq \Sym(\dB\times I)
\end{align*}
for each $i,j\in I$.
\end{defi}

\begin{rmk}\label{rmk:permutationcomponent}
It is evident that for each $(a,i)\in\dA\times I$, there are $\sigma^\dB_{(a,i),0},\sigma^\dB_{(a,i),1}\in\Sym(\dB)$ with
\begin{equation}\label{wreathperm1}
\sigma^\dB_{(a,i)} = ((\sigma^\dB_{(a,i),0},\sigma^\dB_{(a,i),1}),\widehat\cdot\,) \in\Sym(\dB)\wr\Sym(I).
\end{equation}
Note that by this notation, $\sigma^\dB_{(a,i),0}$ and $\sigma^\dB_{(a,i),1}$ are, up to identification of $\dB$ with $\dB_0$ or $\dB_1$, the same as $t_{(a,i)}^0\circ(t_{(a,i)}^1)^{-1}$ and $t_{(a,i)}^1\circ(t_{(a,i)}^0)^{-1}$ respectively. In particular, $\sigma^\dB_{(a,i),0}$ and $\sigma^\dB_{(a,i),1}$ are inverse to each other. Similarly, for each $(b,j)\in\dB\times I$, there are $\sigma^\dA_{(b,j),0},\sigma^\dA_{(b,j),1}\in\Sym(\dA)$ with
\begin{equation}\label{wreathperm2}
\sigma^\dA_{(b,j)} = ((\sigma^\dA_{(b,j),0},\sigma^\dA_{(b,j),1}),\widehat\cdot\,) \in\Sym(\dA)\wr\Sym(I).
\end{equation}
Here again we have $(\sigma^\dA_{(b,j),0})^{-1} = \sigma^\dA_{(b,j),1}$. Alternatively, we can define $\sigma^\dA_{(b,j)}\in\Sym(\dA)\wr\Sym(I)$ and $\sigma^\dB_{(a,i)}\in\Sym(\dB)\wr\Sym(I)$ by \eqref{wreathperm1} and $\eqref{wreathperm2}$, where $\sigma^\dA_{(b,j),i}\in\Sym(\dA)$ and $\sigma^\dB_{(a,i),j}\in\Sym(\dB)$ are such that the square\vspace{-1ex}
\begin{equation}\label{fig:square}
\begin{tikzpicture}[font=\scriptsize,baseline=0]
\coordinate[label=below left:\rlap{$s_{ij}$}\phantom{sss}] (s00) at (0,-0.6);
\coordinate[label=above left:\rlap{$s_{i\widehat{j}}$}\phantom{sss}] (s01) at (0,1);
\coordinate[label=below right:\phantom{sss}\llap{$s_{\widehat{i}j}$}] (s10) at (1.6,-0.6);
\coordinate[label=above right:\phantom{sss}\llap{$s_{\widehat{i}\widehat{j}}$}] (s11) at (1.6,1);
\foreach \point in {s00,s01,s10,s11} \fill [black] (\point) circle (1pt);
\path [-] (s00) edge node [left] {$a$} (s01) edge node [below=2pt] {$b$} (s10) (s11) edge node [above=1pt] {\tiny $\sigma^\dB_{(a,i),\widehat{j}}(b)$} (s01) edge node [right] {\tiny $\sigma^\dA_{(b,j),\widehat{i}}(a)$} (s10) ;
\end{tikzpicture}\vspace{-1ex}
\end{equation}
is a square in $\Sigma_{\dA,\dB}$ for all $a\in\dA$, $b\in\dB$ and $i,j\in I$.
\end{rmk}

We collect some properties of the local permutation groups. These will be useful later in computing the Albanese variety of our fake quadric, compare Propositions \ref{prop:localpermgrpforquarternion} and \ref{prop:trivialalb}. For this we define
\[
P^\dA := \bigl\{ ((\sigma,\tau^\dA\sigma\tau^\dA),\ep) \mid \sigma\in\Sym(\dA), \ep\in\Sym(I) \bigr\} \cong\Sym(\dA)\rtimes\{\pm1\},
\]
where $\tau^\dA:=(-)^{-1}\in\Sym(\dA)$ and the semidirect product structure on $\Sym(\dA)\rtimes\{\pm1\}$ is given by sending $-1$ to the conjugation with $\tau^\dA$ on $\Sym(\dA)$, as a subgroup of $\Sym(\dA)\wr\Sym(I)$ and $P^\dB\leq\Sym(\dB)\wr\Sym(I)$ in the similar way.

\begin{prop}\label{prop:localpermgrp-inversestable}
Let $(\dA,\dB)$ be an inverse-stable $V_4$-structure of a group $\Lambda$. Then $P^\dA_0$ concides with $P^\dA_1$ and is contained in $P^\dA$. The same holds for $\dB$ instead of $\dA$.
\end{prop}

\begin{proof}
Since $(\dA,\dB)$ is inverse-stable, the existence of the square \eqref{fig:square} implies that the square
\begin{equation} \label{fig:squareinv}
\begin{tikzpicture}[font=\scriptsize,baseline=0]
\coordinate[label=below left:\rlap{$s_{ij}$}\phantom{sss}] (s00) at (0,-0.6);
\coordinate[label=above left:\rlap{$s_{i\widehat{j}}$}\phantom{sss}] (s01) at (0,1);
\coordinate[label=below right:\phantom{sss}\llap{$s_{\widehat{i}j}$}] (s10) at (1.6,-0.6);
\coordinate[label=above right:\phantom{sss}\llap{$s_{\widehat{i}\widehat{j}}$}] (s11) at (1.6,1);
\foreach \point in {s00,s01,s10,s11} \fill [black] (\point) circle (1pt);
\path [-] (s00) edge node [left] {$a^{-1}$} (s01) edge node [below=2pt] {$b^{-1}$} (s10) (s11) edge node [above=1pt] {\tiny\,$\sigma^\dB_{(a,i),\widehat{j}}(b)^{-1}$} (s01) edge node [right] {\tiny $\sigma^\dA_{(b,j),\widehat{i}}(a)^{-1}$} (s10) ;
\end{tikzpicture}\vspace{-1ex}
\end{equation}
also belongs to $\Sigma_{\dA,\dB}$. Applying $\gamma_v$ to the square \eqref{fig:squareinv}, we obtain the square
\begin{equation*}
\begin{tikzpicture}[font=\scriptsize,baseline=0]
\coordinate[label=below left:\rlap{$s_{ij}$}\phantom{sss}] (s00) at (0,-0.6);
\coordinate[label=above left:\rlap{$s_{i\widehat{j}}$}\phantom{sss}] (s01) at (0,1);
\coordinate[label=below right:\phantom{sss}\llap{$s_{\widehat{i}j}$}] (s10) at (1.6,-0.6);
\coordinate[label=above right:\phantom{sss}\llap{$s_{\widehat{i}\widehat{j}}$}] (s11) at (1.6,1);
\foreach \point in {s00,s01,s10,s11} \fill [black] (\point) circle (1pt);
\path [-] (s00) edge node [left] {$a$} (s01) edge node [below=2pt] {\tiny\,$\sigma^\dB_{(a,i),\widehat{j}}(b)^{-1}$} (s10) (s11) edge node [above=1pt] {$b^{-1}$} (s01) edge node [right] {\tiny $\sigma^\dA_{(b,j),\widehat{i}}(a)$} (s10) ;
\end{tikzpicture}.\vspace{-1ex}
\end{equation*}
Hence we get $\sigma^\dA_{(b^{-1},\widehat{j}),\widehat{i}}(a) = \sigma^\dA_{(b,j),\widehat{i}}(a)$, implying that $\sigma^\dA_{(b^{-1},\widehat{j})}=\sigma^\dA_{(b,j)}$. Thus $P_0^\dA = P_1^\dA$ since they are generated by the same elements. Now applying $\gamma_h$ to the square \eqref{fig:squareinv}, we obtain the square
\begin{equation*} 
\begin{tikzpicture}[font=\scriptsize,baseline=0]
\coordinate[label=below left:\rlap{$s_{ij}$}\phantom{sss}] (s00) at (0,-0.6);
\coordinate[label=above left:\rlap{$s_{i\widehat{j}}$}\phantom{sss}] (s01) at (0,1);
\coordinate[label=below right:\phantom{sss}\llap{$s_{\widehat{i}j}$}] (s10) at (1.6,-0.6);
\coordinate[label=above right:\phantom{sss}\llap{$s_{\widehat{i}\widehat{j}}$}] (s11) at (1.6,1);
\foreach \point in {s00,s01,s10,s11} \fill [black] (\point) circle (1pt);
\path [-] (s00) edge node [left] {\tiny $\sigma^\dA_{(b,j),\widehat{i}}(a)^{-1}$} (s01) edge node [below=2pt] {$b$} (s10) (s11) edge node [above=1pt] {\tiny $\sigma^\dB_{(a,i),\widehat{j}}(b)$} (s01) edge node [right] {$a^{-1}$} (s10) ;
\end{tikzpicture}.\vspace{-1ex}
\end{equation*}
Therefore, $\sigma^\dA_{(b,j),i}(a^{-1}) = \sigma^\dA_{(b,j),\widehat{i}}(a)^{-1}$, i.e.~$\sigma^\dA_{(b,j),i}\tau^\dA = \tau^\dA\sigma^\dA_{(b,j),\widehat{i}}$. Hence all $\sigma^\dA_{(b,j)}$'s, thus also $P_0^\dA$ and $P_1^\dA$, are contained in $P^\dA$. The statements for $P^\dB,P_0^\dB,P_1^\dB$ can be proved similarly.
\end{proof}

\subsection{The orbispace fundamental group}
We fix the notion of the fundamental group of a quotient orbispace. This occurred in \cite{rhodes} in the context of transformation groups and is a special case of an orbispace in the general context. Here a quotient orbispace is the quotient of a topological space by a group action in the sense of stacks.

\smallskip

In what follows, let $S$ be a topological space with a continuous action of a discrete group $G$. The composition of two paths $\alpha$ and $\beta$ on $S$ such that $\alpha$ ends at the starting point of $\beta$, denoted by $\alpha\ast\beta$, is defined as the path that goes first through $\alpha$ and then through $\beta$. The reverse path of $\alpha$ will be denoted by $\bar\alpha$.

\begin{defi}
Let $s$ be a point in $S$.
\begin{enumerate}
\setlength{\itemsep}{0.4ex}
\item A \textbf{loop of order $g\in G$} with basepoint $s$ is a path from $s$ to $gs$. The homotopy class of a loop $\alpha$ of order $g$ consists of the homotopy class of the path $\alpha$ with the fixed endpoints $s$ and $gs$ will be denoted by $[\alpha;g]$.
\item The \textbf{fundamental group of the quotient orbispace of $S$ by $G$ with basepoint $s$,} denoted by $\pi_1^\orb(S,G,s)$, is defined as the set of homotopy classes of loops of any order with basepoint $s$, together with the composition law
\[
[\alpha;g]\ast[\beta;h] = [\alpha\ast(g\beta);gh] \quad \text{for} \quad [\alpha;g],[\beta;h]\in\pi_1^\orb(S,G,s).
\]
\end{enumerate}
\end{defi}

The most important fact for the orbispace fundamental groups is the following exact sequence:

\begin{prop}\label{prop:orbpi1Htoorbpi1G}
Suppose that the orbit $Gs$ lies in the same path component of $S$ and $H\unlhd G$ is a normal subgroup. Then there is an exact sequence
\[
1\longto\pi_1^\orb(S,H,s)\longto\pi_1^{\mathrm{orb}}(S,G,s)\longto G/H\longto1.
\]
\end{prop}
An immediate consequence is that by setting $H=1$, we obtain the exact sequence
\[
1\longto\pi_1(S,s)\longto\pi_1^{\mathrm{orb}}(S,G,s)\longto G\longto1,
\]
which will play an important role in computing the orbispace fundamental group.
\begin{proof}
By assumption, there exists a path from $s$ to $gs$ for each $g\in G$. This implies that the projection $\pi_1^\orb(S,G,s)\to G,\,[\alpha;g]\mapsto g$ is surjective, thus also $\pi_1^\orb(S,G,s)\to G/H$. Furthermore, $[\alpha;g]\in\pi_1^\orb(S,G,s)$ belongs to its kernel if and only if $g\in H$, i.e. $\alpha\in\pi_1^\orb(S,H,s)$.
\end{proof}

Similarly to the fiber of a covering space, one can define the orbital fiber with the action of the orbispace fundamental group as follows:

\begin{defi} \label{def:orbifiber}
The \textbf{orbital fiber} of a covering space $p:Y\to S$ over $s\in S$ with respect to the $G$-action on $S$ is defined by
\[
p^{-1}(s)_G := \{(y,g)\in Y\times G ~\mid~ p(y)=gs\}.
\]
\end{defi}

\begin{prop}\label{prop:actionpi1orb}
There is a right group action of $\pi_1^\orb(S,G,s)$ on $p^{-1}(s)_G$ given by
\begin{equation*}
(y,g)\cdot[\alpha;g'] := (y.(g\alpha),gg')
\end{equation*}
for $(y,g)\in p^{-1}(s)_G$ and $[\alpha;g']\in\pi_1^\orb(S,G,s)$, where $y.(g\alpha)$ denotes the endpoint of the lifting of $g\alpha$ beginning at $y\in Y$. It is transitive if $Y$ is path-connected, and even simply transitive if $Y$ is simply connected.
\end{prop}

\begin{proof}
A direct calculation shows that this is indeed a well-defined right group action. If $Y$ is path-connected, then for $y_0\in p^{-1}(s)$ and any $(y,g)\in p^{-1}(s)_G$, there exists a path $\gamma$ on $Y$ from $y_0$ to $y$. Hence $(y,g) = (y_0,1)\cdot[(p\circ\gamma);g]$, i.e.~the action is transitive.

\smallskip

If $Y$ is simply connected and $[\alpha;g],[\beta;h]\in\pi_1^\orb(S,G,s)$ satisfy $(y_0,1)\cdot[\alpha;g] = (y_0,1)\cdot[\beta;h]$, then $y_0.\alpha = y_0.\beta$ and $g=h$. Consequently, $\alpha$ and $\beta$ are homotopic since their liftings in $Y$ have the same start point and end point. Thus $[\alpha;g]=[\beta;h]$, i.e.~the action is simply transitive.
\end{proof}

\begin{rmk} \label{rmk:fiberfortrivialstab}
If the stabilizer of $s$ is trivial, there is a bijection between $p^{-1}(s)_G$ and $p^{-1}(Gs)$ given by the projection onto the first component and making $\pi_1^\orb(S,G,s)$ act on $p^{-1}(Gs)$ in natural way.
\end{rmk}

Now suppose that $S$ is path-connected and locally path-connected and let $G_0$ be the kernel of the map $G\to\Aut(S)$ obtained by the $G$-action on $S$. The notion of deck transformations of a covering space over $S$ can be extended to those with respect to the $G$-action on $S$ as follows:

\begin{defi}
For a covering space $p: Y\to S$, we define
\[
\Aut(Y/(S,G)) = \{(\phi,g)\in\Aut(Y)\times G \mid p\circ\phi=gp\}.
\]
A direct calculation shows that $\Aut(Y/(S,G))$ is indeed a group. An element of $\Aut(Y/(S,G))$ is called a \textbf{deck transformation} of $p$ w.r.t.~the $G$-action on $S$.
\end{defi}

\begin{prop}
Let $p:Y\to S$ be a covering space and $s\in S$. Then the group $\Aut(Y/(S,G))$ acts from the left on the orbital fiber $p^{-1}(s)_G$ by
\[
(\phi,g)(y,h):=(\phi(y),gh) \quad \text{for all} ~ (\phi,g)\in\Aut(\tilde{S}/(S,G)) ~\text{and} ~ (y,h)\in p^{-1}(s)_G.
\]
This action is compatible with the right one of $\pi_1^\orb(S,G,s)$ from Proposition \ref{prop:actionpi1orb}.
\end{prop}

\begin{proof}
A direct calculation shows that this is a well-defined left group action. To see the compatibility, let $(\phi,g) \in \Aut(Y/(S,G))$, $(y,h)\in p^{-1}(s)_G$ and $[\alpha;k]\in\pi_1^\orb(S,G,s)$. Then
\begin{align*}
(\phi,g) ((y,h)\cdot[\alpha;k]) &= (\phi,g) (y.(h\alpha),hk) = (\phi(y.(h\alpha)),ghk) \quad \text{and}\\
((\phi,g)(y,h))\cdot[\alpha;k] &= (\phi(y),gh)\cdot[\alpha;k] = (\phi(y).(gh\alpha),ghk).
\end{align*}
Since the path $\phi\circ\gamma$, where $\gamma$ is the lifting of $h\alpha$ starting at $y$, is the lifting of $gh\alpha$ starting at $\phi(y)$, we obtain $\phi(y.(h\alpha)) = \phi(y).(gh\alpha)$, which completes the proof.
\end{proof}

\begin{prop} \label{prop:univdeckvspi1}
Let $p:\tilde{S}\to S$ be the universal covering, $s\in S$ and $\tilde{s}\in p^{-1}(s)$. Then
\[
\Phi: \Aut(\tilde{S}/(S,G))\longto\pi_1^\orb(S,G,s), ~~ (\phi,g) \longmapsto [\alpha;g] ~~ \text{with} ~~ (\phi,g)(\tilde{s},1) = (\tilde{s},1)\cdot[\alpha;g]
\]
is a well-defined group isomorphism.
\end{prop}

\begin{proof}
Since $\tilde{S}$ is simply connected, $\Phi$ is well-defined by Proposition \ref{prop:actionpi1orb}. A calculation shows that it is a group homomorphism. To see the bijectivity, observe that for every $(y,g)\in p^{-1}(s)_G$, there exist unique maps $\phi,\phi'$ making the following diagrams commutative:
\[
\begin{tikzpicture}[baseline=0,descr/.style={fill=white,fill opacity=0.8,inner sep=0.25pt}]
\matrix (m) [matrix of math nodes, row sep=2em, column sep=3em, text height=1.5ex, text depth=0.25ex]
{(\tilde{S},\tilde{s})&(\tilde{S},y)\\(S,s)&(S,gs)\\};
\path[->,font=\scriptsize] (m-1-1) edge node [auto] {$p$} (m-2-1) (m-1-2) edge node [auto] {$p$} (m-2-2) (m-2-1) edge node [auto] {$g$} (m-2-2);
\path[->,dashed,font=\scriptsize] (m-1-1) edge node [auto] {$\phi$} (m-1-2);
\end{tikzpicture}
\qquad \text{and} \qquad
\begin{tikzpicture}[baseline=0,descr/.style={fill=white,fill opacity=0.8,inner sep=0.25pt}]
\matrix (m) [matrix of math nodes, row sep=2em, column sep=3em, text height=1.5ex, text depth=0.25ex]
{(\tilde{S},y)&(\tilde{S},\tilde{s})\\(S,gs)&(S,s)\rlap{.}\\};
\path[->,font=\scriptsize] (m-1-1) edge node [auto] {$p$} (m-2-1) (m-1-2) edge node [auto] {$p$} (m-2-2) (m-2-1) edge node [auto] {$g^{-1}$} (m-2-2);
\path[->,dashed,font=\scriptsize] (m-1-1) edge node [auto] {$\phi'$} (m-1-2);
\end{tikzpicture}
\]
By the uniqueness in the lifting property, $\phi$ and $\phi'$ are inverse to each other. In particular, $(\phi,g)$ is the unique element in $\Aut(\tilde{S}/(S,G))$ sending $(\tilde{s},1)$ to $(y,g)$, i.e.~$\Aut(\tilde{S}/(S,G))$ acts on $p^{-1}(s)_G$ simply transitively. Hence $\Phi$ is bijective as desired.
\end{proof}

\subsection{Computing the orbispace fundamental group}
\begin{defi} \label{loopsinSAB}
In the square complex $\Sigma_{\dA,\dB}$, we define the following paths:
\begin{itemize}
\setlength{\itemsep}{0mm}
\item
For each $a\in\dA$ and $i\in\{0,1\}$, $\delta_{i,a}$ is the path along the edge $(a,i)$.
\item
For each $b\in\dB$ and $j\in\{0,1\}$, $\varepsilon_{j,b}$ is the path along the edge $(b,j)$.
\end{itemize}
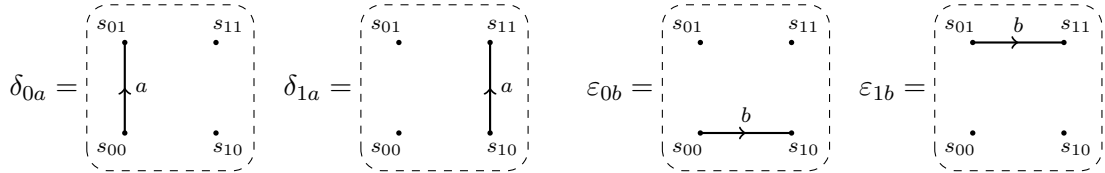
\begin{gather*}
\delta_{0a} = \begin{tikzpicture}[baseline=0, font=\scriptsize] \squaresetting{1pt} \draw[->-=.5,thick] (s00) to node[right] {$a$} (s01); \draw[rounded corners=8pt, dashed] ($(s00)-(0.5,0.5)$) rectangle ($(s11)+(0.5,0.5)$);\end{tikzpicture} \quad
\delta_{1a} = \begin{tikzpicture}[baseline=0, font=\scriptsize] \squaresetting{1pt} \draw[->-=.5,thick] (s10) to node[right] {$a$} (s11); \draw[rounded corners=8pt, dashed] ($(s00)-(0.5,0.5)$) rectangle ($(s11)+(0.5,0.5)$);\end{tikzpicture} \qquad
\varepsilon_{0b} = \begin{tikzpicture}[baseline=0, font=\scriptsize] \squaresetting{1pt} \draw[->-=.5,thick] (s00) to node[above] {$b$} (s10); \draw[rounded corners=8pt, dashed] ($(s00)-(0.5,0.5)$) rectangle ($(s11)+(0.5,0.5)$);\end{tikzpicture} \quad
\varepsilon_{1b} = \begin{tikzpicture}[baseline=0, font=\scriptsize] \squaresetting{1pt} \draw[->-=.5,thick] (s01) to node[above] {$b$} (s11); \draw[rounded corners=8pt, dashed] ($(s00)-(0.5,0.5)$) rectangle ($(s11)+(0.5,0.5)$);\end{tikzpicture}
\end{gather*}
\end{defi}

\begin{prop}\label{prop:pi1orbSAB}
Let $(\dA,\dB)$ be a $V_4$-structure of a group $\Lambda$, and let $\Sigma=\Sigma_{\dA,\dB}$ be the square complex with the $V_4$-action as in Definitions \ref{def:squarecomplex} and \ref{def:V4action}. Furthermore, define
$$ \alpha_a := [\delta_{0,a},\gamma_v], ~ \beta_b := [\varepsilon_{0,b},\gamma_h] \in \pi_1^\orb(\Sigma,V_4,s_{00}) \quad \text{for each} ~ a\in\dA, ~b\in\dB. $$
Then the orbispace fundamental group $\pi_1^\orb(\Sigma,V_4,s_{00})$ has the following presentation:
\begin{equation}\label{presentation-pi1orbSAB}
\pi_1^\orb(\Sigma,V_4,s_{00}) = \left\langle \begin{array}{c} \alpha_a,\beta_b \\ \text{\small for $a\in\dA$, $b\in\dB$} \end{array} \left | \begin{array}{c} \alpha_a\beta_{b'} = \beta_{b}\alpha_{a'} ~ \text{for $a,a'\in\dA$, $b,b'\in\dB$} \\ \text{\small s.t. $ab'=ba'$ in $\Lambda$,} \\ \alpha_{a}\alpha_{a^{-1}}=\beta_{b}\beta_{b^{-1}}=1 ~ \text{for $a\in\dA$, $b\in\dB$} \end{array} \right. \right\rangle.
\end{equation}
\end{prop}

\begin{proof}
We first fix a square $[c,d';d,c']$ in $\Sigma_{\dA,\dB}$ and define the following loops:
\begin{gather*}
x_{0a} = \delta_{0,a}\ast\overline{\delta_{0,c}}, ~ x_{1a} = \varepsilon_{0,d}\ast\delta_{1,a}\ast\overline{\delta_{1,c'}}\ast\overline{\varepsilon_{0,d}}, ~ y_{0b} = \varepsilon_{0,b}\ast\overline{\varepsilon_{0,d}}, ~ y_{1b} = \delta_{0,c}\ast\varepsilon_{1,b}\ast\overline{\varepsilon_{1,d'}}\ast\overline{\delta_{0,c}}.
\end{gather*}
Using Seifert--van Kampen theorem, we obtain the following presentation for $\pi_1(\Sigma,s_{00})$:
\begin{equation*}
\pi_1(\Sigma,s_{00}) = \left\langle \begin{array}{c} x_{0a},x_{1a},y_{0b},y_{1b} \\ \text{\footnotesize for $a\in\dA$, $b\in\dB$} \end{array} \left | \begin{array}{c} x_{0a}y_{1b'} = y_{0b}x_{1a'} \\ \text{\footnotesize for $a,a'\in\dA$, $b,b'\in\dB$ s.t. $ab'=ba'$ in $\Lambda$} \\ x_{0c}=x_{1c'}=y_{0d}=y_{1d'}=1 \end{array} \right. \right\rangle.
\end{equation*}
We now determine a presentation for $\pi_1^\orb(\Sigma,V_4,s_{00})$ using the exact sequence from Proposition \ref{prop:orbpi1Htoorbpi1G}. For this we choose $\alpha_c,\beta_d,\alpha_c\beta_d\in\pi_1^\orb(\Sigma,V_4,s_{00})$ as liftings of $\gamma_v,\gamma_h,\gamma_r\in V_4$ respectively. Conjugating these liftings with the generators of $\pi_1(\Sigma,s_{00})$ yields the following relations for all $a\in\dA$ and $b\in\dB$:
\begin{subequations}
\begin{align}
\label{alpha:x0a}\alpha_c x_{0a} \alpha_c^{-1} &= x_{0,a^{-1}}^{-1}x_{0,c^{-1}}, \\
\label{alpha:x1a}\alpha_c x_{1a} \alpha_c^{-1} &= y_{1d}x_{1,a^{-1}}^{-1}x_{1,(c')^{-1}}y_{1d}^{-1}, \\
\label{alpha:y0b}\alpha_c y_{0b} \alpha_c^{-1} &= y_{1b}y_{1d}^{-1}, \\
\label{alpha:y1b}\alpha_c y_{1b} \alpha_c^{-1} &= x_{0,c^{-1}}^{-1} y_{0b} y_{0d'}^{-1} x_{0,c^{-1}}, \\[1ex]
\label{beta:x0a}\beta_d x_{0a} \beta_d^{-1} &= x_{1a}x_{1c}^{-1},\\
\label{beta:x1a}\beta_d x_{1a} \beta_d^{-1} &= y_{0,d^{-1}}x_{0a}x_{0c'}^{-1}y_{0,d^{-1}},\\
\label{beta:y0b}\beta_d y_{0b} \beta_d^{-1} &= y_{0,b^{-1}}^{-1}y_{0,d^{-1}},\\
\label{beta:y1b}\beta_d y_{1b} \beta_d^{-1} &= x_{1c}y_{1,b^{-1}}^{-1}y_{1,(d')^{-1}}x_{1c}^{-1}.
\end{align}
\end{subequations}
Furthermore, we obtain the following relations by computing the $2$-cocycles $x(\sigma,\tau) = \widehat{\sigma}\widehat{\tau}\widehat{\sigma\tau}^{-1}$:
\begin{align*}
x(\gamma_v,\gamma_v) &= \rlap{$\alpha_c^2 = x_{0,c^{-1}}^{-1}$,}\phantom{\beta_d(\alpha_c\beta_d)\alpha_c^{-1} = \beta_d x(\gamma_h,\gamma_v)^{-1} x(\gamma_h,\gamma_h) \beta_d^{-1}} \\
x(\gamma_v,\gamma_h) &= 1, \\
x(\gamma_v,\gamma_r) &= \alpha_c(\alpha_c\beta_d)\beta_d^{-1} = \alpha_c^2 = x(\gamma_v,\gamma_v), \\
x(\gamma_h,\gamma_v) &= \beta_d\alpha_c(\alpha_c\beta_d)^{-1} = x_{1c}y_{1d}^{-1}, \\
x(\gamma_h,\gamma_h) &= \beta_d^2 = y_{0,d^{-1}}^{-1}, \\
x(\gamma_h,\gamma_r) &= \beta_d(\alpha_c\beta_d)\alpha_c^{-1} = \beta_d x(\gamma_h,\gamma_v)^{-1} x(\gamma_h,\gamma_h) \beta_d^{-1}, \\
x(\gamma_r,\gamma_v) &= \alpha_c\beta_d\alpha_c\beta_d^{-1} = \alpha_c x(\gamma_h,\gamma_v) x(\gamma_v,\gamma_v) \alpha_c^{-1}, \\
x(\gamma_r,\gamma_h) &= (\alpha_c\beta_d)\beta_d\alpha_c^{-1} = \alpha_c x(\gamma_h,\gamma_h) \alpha_c^{-1}, \\
x(\gamma_r,\gamma_r) &= \alpha_c\beta_d\alpha_c\beta_d = x(\gamma_r,\gamma_v) x(\gamma_h,\gamma_h).
\end{align*}
Hence the group $\pi_1^\orb(\Sigma,V_4,s_{00})$ has the following presentation:
\[
\pi_1^\orb(\Sigma,V_4,s_{00}) = \left\langle \begin{array}{c} x_{0a},x_{1a},y_{0b},y_{1b} \\ \text{\small for $a\in\dA$, $b\in\dB$} \\ \alpha_c,\beta_d \end{array} \left |
\begin{array}{c} x_{0a}y_{1b'} = y_{0b}x_{0a'} ~ \text{\small for $a,a'\in\dA$, $b,b'\in\dB$ s.t. $ab'=ba'$} \\
x_{0c}=x_{1c'}=y_{0d}=y_{1d'}=1, ~ \beta_d\alpha_c\beta_d^{-1}\alpha_c^{-1} = x_{1c}y_{1d}^{-1} \\
\alpha_c^2 = x_{0,c^{-1}}^{-1}, ~ \beta_d^2 = y_{0,d^{-1}}^{-1} ~ \text{and \eqref{alpha:x0a} -- \eqref{beta:y1b} hold} \end{array} \right. \right\rangle.
\]
To express a presentation of $\pi_1^\orb(\Sigma,V_4,s_{00})$ in terms of $\alpha_a$'s and $\beta_b$'s for $a\in\dA$ and $b\in\dB$, observe that $x_{0a},x_{1a}$ for $a\in\dA$ and $y_{0b},y_{1b}$ for $b\in\dB$ can be expressed in terms of $\alpha_a$'s and $\beta_b$'s as follows:
\begin{align*}
x_{0a} &= \alpha_a \alpha_c^{-1},\\
x_{1a} &= \beta_d\alpha_a\alpha_{c'}^{-1}\beta_d^{-1},\\
y_{0b} &= \beta_b\beta_d^{-1}, \qquad \text{and}\\
y_{1b} &= \alpha_c\beta_b\beta_{d'}^{-1}\alpha_c^{-1}.
\end{align*}
Hence $\alpha_a$'s and $\beta_b$'s generate the group $\pi_1^\orb(\Sigma,V_4,s_{00})$. Elementary computations show that the relations above are equivalent to the relations given in \eqref{presentation-pi1orbSAB}, which completes the proof.
\end{proof}

\begin{rmk} \label{pi1orbSABreduced}
The presentation in \eqref{presentation-pi1orbSAB} can be reduced by observing that the relation $\alpha_a\beta_{b'}=\beta_b\alpha_{a'}$ for the square $[a,b';b,a']\in\bS(\Sigma_{\dA,\dB})$ is, in virtue of $\alpha_{a^{-1}}=\alpha_a^{-1}$ and $\beta_{b^{-1}}=\beta_b^{-1}$, equivalent to the relations for the other squares in the $V_4$-orbit of $[a,b';b,a']$. Hence if $\bar\bS(\Sigma_{\dA,\dB})$ is a set of representatives of the $V_4$-orbits of squares in $\Sigma_{\dA,\dB}$, then
\[
\pi_1^\orb(\Sigma,V_4,s_{00}) = \left\langle \begin{array}{c} \alpha_a,\beta_b \\ \text{\small for $a\in\dA$, $b\in\dB$} \end{array} \left | \begin{array}{c} \alpha_a\beta_{b'} = \beta_{b}\alpha_{a'} ~ \text{for $[a,b';b,a']\in\bar\bS(\Sigma_{\dA,\dB})$,} \\ \alpha_{a}\alpha_{a^{-1}}=\beta_{b}\beta_{b^{-1}}=1 ~ \text{for $a\in\dA$, $b\in\dB$} \end{array} \right. \right\rangle.
\]
Furthermore, we can reduce the number of generators by taking only one of $\alpha_a$ and $\alpha_{a^{-1}}$ for $a\in\dA$ such that $a\neq a^{-1}$ and one of $\beta_b$ and $\beta_{b^{-1}}$ for $b\in B$ such that $b\neq b^{-1}$ and replacing $\alpha_{a^{-1}}$ by $\alpha_a^{-1}$ and $\beta_{b^{-1}}$ by $\beta_b^{-1}$ in the relations obtained from the squares in $\bar\bS(\Sigma_{\dA,\dB})$ if necessary.
\end{rmk}

\subsection{Comparison Theorem}
Let $(\dA,\dB)$ be a $V_4$-structure of a group $\Lambda$. Our goal is to derive an isomorphism between $\pi_1^\orb(\Sigma_{\dA,\dB},V_4,s_{00})$ and $\Lambda$ under certain conditions by comparing their actions on a product of two trees.

\begin{prop}\label{cor:univcoverSAB}
Let $(\dA,\dB)$ be as above, $m:=\#\dA$ and $n:=\#\dB$. The universal covering of $\Sigma_{\dA,\dB}$ is $T_m\times T_n$.
\end{prop}

\begin{proof}
This follows from Lemma \ref{lem:links} and the fact that the universal covering of a square complex is a product of trees if and only if the links of its vertices are all complete bipartite graphs, see \cite[Thm.3.8]{squarecomplex}.
\end{proof}

\begin{thm}\label{thm:compareonTxT}
Let $\Lambda$ be a group, $(\dA,\dB)$ be a $V_4$-structure of $\Lambda$, $m:=\#\dA$ and $n:=\#\dB$. Suppose that $\Lambda$ acts on the product $Y:=T_m\times T_n$ in such a way that the group action respects the cellular structure and for a distinguished vertex $v\in Y$, the orbits $\dA.v$ and $\dB.v$ agree with the set of vertical and horizontal neighbors of $v$ respectively. Then the following holds:
\begin{enumerate}
\item
$\Lambda$ acts simply transitively on the vertices of $Y$.
\item
The group homomorphism
\begin{equation*}
\varphi:\pi_1^\orb(\Sigma_{\dA,\dB},V_4,s_{00})\longto \Lambda, ~ \alpha_a\longmapsto a ~ \text{and} ~ \beta_b\longmapsto b,
\end{equation*}
is an isomorphism, under which $\pi_1(\Sigma_{\dA,\dB},s_{00})$ is identified with a normal subgroup $\Gamma\unlhd \Lambda$ with $V_4$ as quotient. The action of $\Gamma$ on $Y$ yields a canonical isomorphism of square complexes $\Gamma\backslash Y \cong \Sigma_{\dA,\dB}$.
\end{enumerate}
\end{thm}

\begin{proof}
We first show that any vertex $w$ of $Y$ lies in the orbit $\Lambda.v$, i.e.~the action of $\Lambda$ on the vertices of $Y$ is transitive. Since $Y$ is connected, we can show this by induction on the distance to $v$. The base step is clear since $v=1.v$. To show the inductive step, observe first that $w$ has a neighbor vertex $w'$ with smaller distance to $v$. By induction hypothesis, we have $w'=g.v$ for some $h\in G$. This implies that $g^{-1}.w$ is a neighbor of $g^{-1}.w'=v$, i.e.~$g^{-1}.w\in\dA.v\cup\dB.v$ by assumption. Consequently, $g^{-1}.w$ and thus also $w$ lie in the orbit $\Lambda.v$.

\smallskip

Now let $\widetilde{\Sigma}$ be the universal covering of $\Sigma_{\dA,\dB}$. It is isomorphic to $T_m\times T_n$ by Proposition \ref{cor:univcoverSAB}. Furthermore, $\Lambda':=\Aut(\widetilde{\Sigma}/(\Sigma_{\dA,\dB},V_4))$ acts simply transitively on the orbital fiber over $s_{00}$, hence also on the vertices of $\widetilde{\Sigma}$ by Remark \ref{rmk:fiberfortrivialstab}. Choosing a distinguished vertex $\tilde{s}\in\widetilde{\Sigma}$ in the fiber of $s_{00}$, we obtain an isomorphism $\Phi:\pi_1^\orb(\Sigma_{\dA,\dB},V_4,s_{00})\xrightarrow{\sim}\Lambda'$ from Proposition \ref{prop:univdeckvspi1}.

\smallskip

We wish to show that $\varphi$ is an isomorphism. Note that $\varphi$ is well-defined as can be easily seen from the presentation of $\pi_1^\orb(\Sigma_{\dA,\dB},V_4,s_{00})$ in \eqref{presentation-pi1orbSAB} and surjective since its image contains $\dA\cup\dB$ which generates $\Lambda$. We consider instead $\varphi':=\varphi\circ\Phi^{-1}:\Lambda'\to\Lambda$ and construct a $\varphi'$-equivariant mapping $f:\widetilde{\Sigma}\to Y$ sending $\tilde{s}$ to $v$ as follows: 

\begin{itemize}
\item
On the $0$-skeleton, $f$ is uniquely determined by $f(t) := \varphi(\gamma).v$ for each vertex $t\in\widetilde{\Sigma}$, where $\gamma$ is a uniquely determined element in $\Lambda'$ such that $\gamma.\tilde{s}=s$.
\item
For the $1$-skeleton, observe the neighbors of a vertex $t=\gamma.\tilde{s}$ are $t\cdot\alpha_a = (\gamma\cdot\Phi(\alpha_a)).\tilde{s}$ for $a\in\dA$ and $t\cdot\beta_b = (\gamma\cdot\Phi(\beta_b)).\tilde{s}$ for $b\in\dB$. Since $f(t\cdot\alpha_a) = \varphi'(\alpha).(a.v)$, $f(t\cdot\beta_b) = \varphi'(\beta).(b.v)$ and neighbors of $v$ are $a.v$ for $a\in\dA$ and $b.v$ for $b\in\dB$, we see that $f$ preserves the adjacency of two vertices. Hence $f$ can be extended to the $1$-skeleton.
\item
For the $2$-skeleton, observe that if $t=\gamma.\tilde{s}$ is a vertex of a square in $\widetilde{\Sigma}$, its other vertices are $t\cdot\alpha_a$, $t\cdot\beta_b$ and $t\cdot(\alpha_a\beta_{b'})=t\cdot(\alpha_a\beta_{b'})$ for some $a,a'\in\dA$ and $b,b'\in\dB$ such that $ab'=ba'$. These are mapped to $\varphi(\gamma).v$, $\varphi(\gamma).(a.v)$, $\varphi(\gamma).(b.v)$ and $\varphi(\gamma).(ab'.v)=\varphi(\gamma).(ba'.v)$, which are again vertices of a square in $Y$. Hence $f$ can be extended to the $2$-skeleton.
\end{itemize}

Observe that $f$ maps the edges resp.~squares attached to $\tilde{s}$ bijectively to the edges resp.~squares attached to $v$ by the assumption on $\dA.v$ and $\dB.v$ and the fact that each square attached to $\tilde{s}$ and $v$ is uniquely determined by its horizontal and vertical edges attached to $\tilde{s}$ and $v$ respectively. Hence the restriction of $f$ to the neighborhood $U$ of $\tilde{s}$ consisting of the interiors of the edges and the squares attached to $\tilde{s}$ yields a homeomorphism to a neighborhood $V$ of $v$. Since the preimage of $v$ contains no vertices of the same square, the preimage of $V$ is a disjoint union of translates of $U$. The translations of $V$ cover the whole $Y$ since $\Lambda$ acts transitively on the vertices of $Y$. Since $\varphi'$ is surjective and $f$ is $\varphi'$-equivariant, $f$ is a covering map.

\smallskip

Since $Y$ as product of two trees is simply connected and $\widetilde{\Sigma}$ is connected, $f$ is a homeomorphism. This implies the injectivity of $\varphi'$ and thus also $\varphi$. Hence $\varphi$ is an isomorphism. Our $\varphi$ maps $\pi_1(\Sigma_{\dA,\dB},s_{00})$ to a normal subgroup $\Gamma\unlhd\Lambda$ with $V_4$ as quotient. Since the action of $\Gamma$ on $Y$ is isomorphic to the action of $\pi_1(\Sigma_{\dA,\dB},s_{00})$ on $\widetilde{\Sigma}$, we get $\Gamma\backslash Y\cong \Sigma_{\dA,\dB}$. Finally, since the action of $\Lambda'$ on the vertices of $\widetilde{\Sigma}$ is simply transitive, so is the action of $\Lambda$ on the vertices of $Y$
\end{proof}

\section{Quaternionic arithmetic lattices} \label{sec:quatlattices}
\subsection{Quaternion algebras in characteristic $2$}
We give here a review on basic facts about quaternion algebras in characteristic $2$ and their arithmetic. The main reference will be \cite{vigneras}.

\smallskip

A \emph{quaternion algebra} over a field $K$ is a central simple $K$-algebra of dimension $4$. It is either isomorphic to the matrix algebra $M_2(K)$ or a division algebra.

\smallskip

Suppose from now on that $K$ is a field of characteristic $2$. Then each quaternion algebra over $K$ can be described as follows:

\begin{nota}
For $a\in K$ and $b\in K^\times$, we define
\[
\left[\frac{a,b}{K}\right):=K\{I,J\}\big/(I^2+I=a, ~ J^2=b, ~ IJ=J(I+1)).
\]
\end{nota}

One can show that the $K$-algebra defined above is indeed a quaternion algebra and conversely that each quaternion algebra is of this form. Its class in $\Br(K)=H^2(K,\bG_m)$ is $2$-torsion.

\smallskip

Every quaternion algebra $Q$ over $K$ has a unique $K$-linear anti-involution $\overline{\,\cdot\,}:Q\to Q$, called \emph{conjugation}, such that the \emph{reduced norm} $\rn(q):=q\bar{q}$ and the \emph{reduced trace} $\tr(q):=q+\bar{q}$ belong to $K$ for all $q\in Q$. If $L/K$ is a field extension such that there exists an injective $K$-algebra homomorphism $\rho:Q\hookrightarrow\rM_2(L)$, then one can show that for all $q\in Q$,
\[
\rho(\bar{q}) = \mathrm{adj}(\rho(q)), \quad \rn(q)=\det(\rho(q)) \quad \text{and} \quad \tr(q)=\tr(\rho(q)).
\]

\smallskip

Assume from now on that $K$ is a global field and $Q$ is a quaternion algebra over $K$. For each place $\fp$ of $K$, let $\nu_\fp$ be the associated normalized valuation on $K$, $K_\fp$ the completion of $K$ at $\fp$ and $\kappa(\fp)$ its residue field.

\begin{defi}
We say that a place $\fp$ of $K$ is \emph{ramified} in $Q$ if $Q\otimes_KK_\fp$ is a division algebra, otherwise \emph{unramified} or \emph{split}.
\end{defi}

\begin{defi}
The \emph{local symbol} on $K$ at a place $\fp$ is defined by
\[
[\cdot,\cdot)_\fp:K \times K^\times\longto\bZ/2\bZ, \quad (a,b) \longmapsto [a,b)_\fp := \begin{cases}0 & \text{if $[\frac{a,b}{K})$ splits in $\fp$,} \\ 1 & \text{otherwise.}\end{cases}
\]
Note that this coincides with the local symbol defined in \cite[Ch.XIV, \textsection5]{serre-localfields}.
\end{defi}

To state the ramification criterion, let $t\in K_\fp$ be a uniformizer. Recall that the \emph{residue} of a differential form $\omega$ on $K_\fp\cong\kappa(\fp)\prr{t}$, denoted by $\res(\omega)\in\kappa(\fp)$, is defined as the coeffient of $t^{-1}$ in $f\in K_\fp$ such that $\omega=f\,dt$.

\begin{prop}\label{prop:residuecriterion}
Let $a\in K$ and $b\in K^\times$. Then
\[
[a,b)_\fp = \tr_{\kappa(\fp)/\bF_2}\left( \res\frac{a\,db}{b} \right).
\]
In particular, the quaternion algebra $\bigl[\frac{a,b}{K}\bigr)$ splits at $\fp$ if and only if $\tr_{\kappa(\fp)/\bF_2}\left( \res\frac{a\,db}{b} \right)=0$.
\end{prop}

\begin{proof}
\cite[Ch.XIV, Prop.15]{serre-localfields}
\end{proof}

From Proposition \ref{prop:residuecriterion}, the quaternion algebra $\bigl[\frac{a,b}{K}\bigr)$ can ramify at most in those places $\fp$ of $K$ such that $\nu_\fp(a)<0$ or $\nu_\fp(b)\neq0$. The set of ramified places of a quaternion algebra $Q$ will be denoted by $\Ram(Q)$.

\smallskip

Now let $S$ be a set of places in $K$ and $R:=O_{K,S}$ be the ring of $S$-integers. Furthermore, let $\fO$ be an $R$-order with an $R$-basis $(u_1,u_2,u_3,u_4)$, i.e.~a subring $\fO\subseteq Q$ such that $u_1,u_2,u_3,u_4$ generate $Q$ as $K$-vector space. Then the ideal generated by
\[
{\rm disc}(u_1,u_2,u_3,u_4) := \det(\tr(u_iu_j))_{i,j}
\]
is independent of the choice of a basis and a square of the \emph{reduced discriminant} of $\fO$, denoted by $d(\fO)$. We have the following criterion for the maximality:

\begin{prop}\label{prop:maxorderoverdedekind}
Let $Q$, $R$, $\fO$ be as above. Then $\fO$ in a maximal order in $Q$ if and only if
\[
d(\fO) = \prod_{\fp\in\mathrm{Ram}(Q)\setminus S}\fp.
\]
\end{prop}

\begin{proof}
\cite[Ch.III, Cor.5.3]{vigneras}
\end{proof}

\subsection{The quaternion algebra $\bigl[\frac{z,1+z^3}{K}\bigr)$ and its lattices}
Let $K:=\bF_2(z)$ be the rational function field over $\bF_2$ in the parameter $z$. We are interested in the quaternion algebra
\[
Q := \left[\frac{z,1+z^3}{K}\right) = K\{I,J\}\big/\{I^2+I=z, ~ J^2=1+z^3, ~ IJ=J(I+1)\}.
\]
In what follows, we shall denote the place of $K$ at a closed point $x\in\bP_{\bF_2}^1$ by $\fp_x$. We also write $\nu_x:=\nu_{\fp_x}$ and $K_x:=K_{\fp_x}$. Furthermore, let $\zeta\in\bar\bF_2$ be a primitive third root of unity. The point in $\bP_{\bF_2}^1$ given by the minimal polynomial of $\zeta$ over $\bF_2$ will be again denoted by $\zeta$.

\begin{lem}\label{whereQramifies}
$Q$ ramifies exactly in $\fp_1$ and $\fp_\zeta$. In particular, $Q$ is a division algebra.
\end{lem}

\begin{proof}
$Q$ ramifies at most in the places $\fp$ where $\nu_\fp(z)<0$ or $\nu_\fp(1+z^3)\neq0$. If $\nu_\fp(z)<0$, then $\fp=\fp_\infty$, i.e. $u:=z^{-1}$ is a uniformizer and we have
$$\frac{z\,d(1+z^3)}{1+z^3} = \frac{u^{-1}\,d(1+u^{-3})}{1+u^{-3}} = \frac{u^2\cdot u^{-4}\,du}{1+u^3} = (u^{-2}+u+u^4+u^7+\cdots)\,du$$
So the residue of this differential form vanishes, implying that $[z,1+z^3)_{K_\infty}=0$ by Proposition \ref{prop:residuecriterion}, i.e. $Q$ splits at $\fp_\infty$.\par\smallskip
If $\nu_\fp(1+z^3)\neq0$ and $\fp\neq\fp_\infty$, then $\fp\in\{\fp_1,\fp_\zeta\}$. In either case, the residue of $\frac{z\,d(1+z^3)}{1+z^3}$ does not vanish since $\nu_\fp(\frac{z}{1+z^3})=-1$. In particular, $[z,1+z^3)_{K_1}\neq 0$, i.e.~$Q$ ramifies in $\fp_1$. Since the number of ramified places is even, $Q$ must also ramify at $\fp_\zeta$.
\end{proof}

We now construct two splittings of $Q$ under quadratic extensions, for which the place $\fp_0$ resp.~$\fp_\infty$ splits. For the first splitting, let $\bF_2(y)$ be the quadratic extension of $K$ defined by
\[
y^2+y=z.
\]
This extension has a non-trivial Galois automorphism given by $y\mapsto 1+y$. The place $\fp_0$ splits since the above equation reduces modulo $z$ to $y(y+1)=y^2+y=0$ over $\bF_2$.

\begin{lem}\label{splitting1}
The quaternion algebra $Q$ splits over $\bF_2(y)$, i.e.~$Q$ has a $2$-dimensional representation $\rho_y:Q\to\ M_2(\bF_2(y))$ over $\bF_2(y)$ given by
\[
I \mapsto \rho_y(I):=\matzz{y}00{1+y} \quad \text{and} \quad J \mapsto \rho_y(J):=\matzz0{1+z^3}10.
\]
\end{lem}

\begin{proof}
This is an elementary computation in $\bF_2(y)$.
\end{proof}

For the second splitting, set $u=z^{-1}\in K$ as before and let $\bF_2(t)$ be the quadratic extension of $K$ defined by
\[
t^2+t=u.
\]
This extension has a non-trivial Galois automorphism given by $t\mapsto 1+t$. By the same argument as before, the place $\fp_\infty$ splits under this extension.

\begin{lem}\label{splitting2}
The quaternion algebra $Q$ splits over $\bF_2(t)$, i.e.~$Q$ has a $2$-dimensional representation $\rho_t:Q\to\ M_2(\bF_2(t))$ over $\bF_2(t)$ given by
\[
I \mapsto \rho_t(I):=\matzz{1+u+t}{1+u^3}{u^{-1}}{u+t} \quad \text{and} \quad J \mapsto \rho_t(J):=\matzz0{u^{-1}+u^2}{u^{-2}}{0}.
\]
\end{lem}
\begin{proof}
This is an elementary computation in $\bF_2(t)$.
\end{proof}

We are interested in arithmetic lattices of the algebraic group over $K$
\[
G := \PGL_{1,Q} = \GL_{1,Q}/\bG_m.
\]
Let $S_0:=\{\fp_0,\fp_\infty\}$, $S_1:=\{\fp_0,\fp_\infty,\fp_1\}$ and $S:=\{\fp_0,\fp_\infty,\fp_1,\fp_\zeta\}$. Note that the $S$-rank of $G$ is $2$ since $G$ remains anisotropic at the places where $Q$ ramifies and each unramified place in $S$ contributes rank $1$. We will be concerned with the $S_0$-, $S_1$- and $S$-arithmetic lattices. For this we define the following rings:
\[
R_0:=O_{K,S_0}=\bF_2[z,\frac{1}{z}], ~~ R_1:=O_{K,S_1}=\bF_2[z,\frac{1}{z(1+z)}] ~~ \text{and} ~~ R:=O_{K,S}=\bF_2[z,\frac{1}{z(1+z^3)}].
\]
To define an $R_0$-integral structure on $G$, consider the $R_0$-order
\[
\fO_0 := R_0\oplus R_0\cdot I \oplus R_0\cdot J \oplus R_0\cdot IJ \subseteq Q.
\]
It is maximal by Proposition \ref{prop:maxorderoverdedekind} since the trace form on the basis $1,I,J,IJ$ has discriminant $(1+z^3)^2$, which is the square of $(1+z^3)=(1+z)(1+z+z^2)$. Consequently, the $R_1$-order $\fO_1:=\fO_0\otimes_{R_0}R_1$ and the $R$-order $\fO:=\fO_0\otimes_{R_0}R$ are maximal over the corresponding rings. The latter one is even an Azumaya algebra since $1+z^3$ is invertible in $R$.

\smallskip

The basis $1,I,J,IJ$ yields an integral structure on $G$ and the following arithmetic subgroups
\[
\fO_0^\times/R_0^\times = G(R_0) \subseteq \fO_1^\times/R_1^\times = G(R_1) \subseteq \fO^\times/R^\times = G(R) \subseteq Q^\times/K^\times = G(K).
\]
Note that the above equalities hold since $R_0$, $R_1$ and $R$ are principal ideal domains.

\begin{lem}
$G(R)$ is a cocompact lattice under the diagonal embedding
\begin{equation}\label{diagembedding}
\rho: G(R) \lhook\joinrel\longrightarrow G(K_0)\times G(K_\infty).
\end{equation}
\end{lem}

\begin{proof}
Since $G$ is connected, reductive, semisimple and anisotropic over $K$, we can use \cite[Ch.I, Thm.3.2.4-5]{margulis} to conclude that the diagonal embedding
\[
G(R) \lhook\joinrel\longrightarrow G(K_0) \times G(K_1) \times G(K_\zeta) \times G(K_\infty)
\]
makes $G(R)$ a cocompact lattice in $G(K_0) \times G(K_1) \times G(K_\zeta) \times G(K_\infty)$. Furthermore, $G(K_1)$ and $G(K_\zeta)$ are compact since $Q$ ramifies at $\fp_1$ and $\fp_\zeta$. Hence the factors $G(K_1)$ and $G(K_\zeta)$ can be omitted to obtain a cocompact lattice under the embedding $G(R) \hookrightarrow G(K_0)\times G(K_\infty)$.
\end{proof}

\subsection{A $V_4$-structure and its group theory}
We introduce a $V_4$-structure of a subgroup of $G(R)$ that will be useful in finding a presentation of $G(R)$ and establish its group theory.

\begin{defi}\label{introduceelements}
In the quaternion algebra $Q$, we define the following elements:
\begin{align*}
C_1 &:= 1+z^2+IJ, \\
C_2 &:= z+z^2+IJ, \\
B_1 &:= (1+z)I+J, \\
B_2 &:= z+z^2+(1+z)I+J+IJ.
\end{align*}
\end{defi}

\begin{lem}
$B_1,B_2,C_1,C_2$ are all invertible in $\fO_1=R_1\oplus R_1\cdot I\oplus R_1\cdot J\oplus R_1\cdot I\!J\subseteq Q$
\end{lem}

\begin{proof}
We compute their reduced norms directly:
\[
\rn(B_1)=1+z, ~ \rn(B_2)=z+z^2, ~ \rn(C_1)=1+z ~ \text{and} ~ \rn(C_2)=z+z^2.
\]
Since their norms are units in $R_1$, they are invertible in $\fO_1$.
\end{proof}

\begin{nota}
The images of $B_1,B_2,C_1,C_2$ in $G(R_1)\subseteq G(K)=Q^\times/K^\times$ will be denoted by $b_1,b_2,c_1,c_2$ respectively. Furthermore, we define $\dA:=\{b_1,b_1^{-1},c_1\}$ and $\dB:=\{b_2,b_2^{-1},c_2\}$.
\end{nota}

\begin{lem}\label{lem:V4inGR}
The pair $(\dA,\dB)$ is an inverse-stable $V_4$-structure of $\Lambda:=\langle b_1,b_2,c_1,c_2 \rangle \leq G(R_1)$.
\end{lem}

\begin{proof}
Notice first that since $C_1^2=1+z$ and $C_2^2=z+z^2$, we have $c_1^2=c_2^2=1$, implying that $\dA$ and $\dB$ are closed under taking inverse. Now we compute the following products:
\begin{align*}
B_1B_2 &= (1+z)(z+z^2+zI+J+IJ) \\
B_1B_2^{-1} &= z^{-1}(z+z^2+I+IJ) \\
B_1C_2 &= (1+z)(1+z+z^2+I+IJ)\\[0.5ex]
B_1^{-1}B_2 &= I\\
B_1^{-1}B_2^{-1} &= (z+z^2)^{-1}(1+z+zI+J)\\
B_1^{-1}C_2 &= 1+I\\[0.5ex]
C_1B_2 &= (1+z)(z^2+zI+J+IJ)\\
C_1B_2^{-1} &= z^{-1}(1+zI+J)\\
C_1C_2 &= (1+z)(z^2+IJ)
\end{align*}
Since none of them is a scalar multiple of any others, the map $\dA\times\dB \to \dA\dB, ~ (g,h)\mapsto gh$ is indeed bijective. Furthermore, we have
\begin{itemize}
\item $B_2B_1 = (1+z)(1+zI+J)$, which implies the relation $b_2b_1 = c_1b_2^{-1}$, and consequently $b_2^{-1}c_1  = b_1b_2$, $b_2c_1 = b_1^{-1}b_2^{-1}$ and $b_2^{-1}b_1^{-1}=c_1b_2$.
\item $C_2B_1 = (1+z)I$, which implies the relation $c_2b_1 = b_1^{-1}b_2$, and consequently $c_2b_1^{-1} = b_1b_2^{-1}$, $b_2^{-1}b_1 = b_1^{-1}c_2$ and $b_2b_1^{-1} = b_1c_2$.
\item $C_2C_1 = (1+z)(z^2+IJ)$, i.e. $c_1c_2=c_2c_1$.
\end{itemize}
Hence we get $\dB\dA=\dA\dB$, which implies that $\dA$ and $\dB$ define a $V_4$-structure of $\Lambda$. It is easily seen from the list of relations that this $V_4$-structure is inverse-stable.
\end{proof}

\begin{prop}\label{prop:localpermgrpforquarternion}
For the $V_4$-structure $(\dA,\dB)$ as in Lemma \ref{lem:V4inGR}, we have
\[
P_0^\dA=P_1^\dA=P^\dA \quad \text{and} \quad P_0^\dB=P_1^\dB=P^\dB,
\]
where $P^\dA$ and $P^\dB$ are as in Proposition \ref{prop:localpermgrp-inversestable}.
\end{prop}

\begin{proof}
We begin by computing the generators of $P^\dA_0$ explicitly and obtain
\begin{align*}
\sigma^\dA_{(b_2,0)} &= (((b_1~c_1~b_1^{-1}),(b_1^{-1}~c_1~b_1)),\widehat{\cdot}\,), \\
\sigma^\dA_{(b_2^{-1},0)} &= (((b_1^{-1}~c_1~b_1),(b_1~c_1~b_1^{-1})),\widehat{\cdot}\,) \quad \text{and} \\
\sigma^\dA_{(c_2,0)} &= (((b_1~b_1^{-1}),(b_1^{-1}~b_1)),\widehat{\cdot}\,).
\end{align*}
From this we obtain
\begin{align*}
\sigma^\dA_{(b_2,0)}\sigma^\dA_{(c_2,0)} &= (((c_1~b_1^{-1}),(c_1~b_1)),\id_I) = \psi_\dA((c_1~b_1^{-1}),+1) \quad \text{and}\\
\sigma^\dA_{(b_2^{-1},0)}\sigma^\dA_{(c_2,0)} &= (((c_1~b_1),(c_1~b_1^{-1})),\id_I) = \psi_\dA((c_1~b_1),+1).
\end{align*}
This means that $\Sym(\dA)$ as a subgroup of $\Sym(\dA)\rtimes\{\pm1\}\cong P^\dA$ is entirely contained in $P_0^\dA=P_1^\dA$. Since $P_0^\dA=P_1^\dA$ also contains at least one image of an element of $\Sym(\dA)\rtimes\{\pm1\}$ with non-trivial projection onto $\{\pm1\}$, e.g.~$\sigma^\dA_{(c_2,0)} = ((b_1~b_1^{-1}),-1)$, it must entirely contain $P^\dA$, i.e.~$P^\dA_0=P^\dA_1=P^\dA$ since the other inclusion is known from Proposition \ref{prop:localpermgrp-inversestable}.

\smallskip

Now for $P^\dB_1$, we have
\begin{align*}
\sigma^\dB_{(b_1,1)} &= (((b_2~c_2~b_2^{-1}),(b_2^{-1}~c_2~b_2)),\widehat{\cdot}\,),\\
\sigma^\dB_{(b_1^{-1},1)} &= (((b_2^{-1}~c_2~b_2),(b_2~c_2~b_2^{-1})),\widehat{\cdot}\,) \quad \text{and}\\
\sigma^\dB_{(c_1,0)} &= (((b_2~b_2^{-1}),(b_2^{-1}~b_2)),\widehat{\cdot}\,).
\end{align*}
Hence the calculation of $P^\dB_0=P^\dB_1$ can be done in the similar way to $P_0^\dA$, implying that $P^\dB_0=P^\dB_1=P^\dB$ as desired.
\end{proof}

\subsection{The Bruhat-Tits action}
For each closed point $x\in\bP_{\bF_2}^1$, let $K_x$ and $\nu_x$ be as before, $O_x\subseteq K_x$ the valuation ring, $\pi_x\in O_x$ a uniformizer and $\kappa(x)=O_x/(\pi_x)$ the residue field.

\smallskip

The Bruhat-Tits tree of $\PGL_2(K_x)$ is defined as follows: The vertices of $T_x$ are homothety classes of $O_x$-lattices in $K_x^2$. Two such classes are linked by an edge if there are representatives $M_1$, $M_2$ such that $\pi M_1\subsetneq M_2\subsetneq M_1$. The action of $\PGL_2(K_x)$ on the homothety classes of $O_x$-lattices in $K_x^2$ is defined by left multiplication. This induces a simplicial action on $T_x$.

\begin{nota}
The \textbf{standard vertex} of $T_x$ is given by $O_x^2\subseteq K_x^2$ and will be denoted by $w_x$.
\end{nota}

\begin{prop}
The Bruhat-Tits tree $T_x$ defined above is an infinite tree of constant valency $\#\kappa(x)+1$. If $[M_1],[M_2]\in\bV(T_x)$, then there exist $a,b\in\bZ$ and a basis $\{m_1,m_2\}$ of $M_1$ such that $\{\pi_x^am_1,\pi_x^bm_2\}$ is a basis of $M_2$, and the distance between $[M_1]$ and $[M_2]$ is given by
\[
d([M_1],[M_2]) = |a-b|.
\]
\end{prop}

\begin{proof}
\cite[Ch.II, \textsection1]{serre-trees}
\end{proof}

This proposition allows us to determine the distance between the standard vertex and its image under the action of any element of $\PGL_2(K_x)$ as follows:

\begin{cor} \label{cor:diststdvertex}
If $A\in\GL_2(K_x)$ has coprime entries in $O_x$, then
\[
d(w_x,Aw_x) = \nu_x(\det(A)).
\]
In particular, for an arbitrary $A\in\GL_2(K_x)$, we have
\[
d(w_x,Aw_x) \equiv \nu_x(\det(A)) \pmod2,
\]
and $\PGL_2(O_x)\leq\PGL_2(K_x)$ is the stabilizer of $w_x$.
\end{cor}

Now we can define the action of $G(R)$ on $T_0\times T_\infty$ component-wise. Under the extension $\bF_2(y)/\bF_2(z)$ as in Lemma \ref{splitting1}, the place $\{z=0\}$ splits into the places defined by $y$ resp.~$1+y$, whence $K_0\cong\bF_2\prr{y}$. The embedding $\rho_y$ from Lemma \ref{splitting1} then induces an isomorphism $G(K_0)\cong\PGL_2(\bF_2\prr{y})$. Hence $G(R)$ acts on the horizontal component $T_0$ via embedding
\[
G(R) \longhookrightarrow G(K_0)\cong\PGL_2(\bF_2\prr{y}).
\]
Similarly, under the extension $\bF_2(t)/\bF_2(z)$ as in Lemma \ref{splitting2}, the place $\{z=\infty\}$ splits into the places defined by $t$ resp.~$1+t$, whence $K_\infty\cong\bF_2\prr{t}$. The embedding $\rho_t$ from Lemma \ref{splitting2} then induces an isomorphism $G(K_\infty)\cong\PGL_2(\bF_2\prr{t})$. Hence $G(R)$ acts on the vertical component $T_\infty$ via embedding
\[
G(R) \longhookrightarrow G(K_\infty)\cong\PGL_2(\bF_2\prr{t}).
\]
Combining the actions of both components, we obtain the Bruhat-Tits action of $G(R)$ on $T_0\times T_\infty$, which will serve us in determining presentations of arithmetic lattices.

\subsection{The stabilizer}
We are going to determine the stabilizer $G(R)_w$ of the standard vertex $w=(w_0,w_\infty)\in T_0\times T_\infty$ under the Bruhat-Tits action. The key is to show that each element has order at most $2$.

\begin{lem}
The stabilizer $G(R)_w$ is a finite group.
\end{lem}
\begin{proof}
We use the embedding $\rho:G(R)\hookrightarrow \PGL_2(\bF_2\prr{y})\times\PGL_2(\bF_2\prr{t})$ and observe that
\[
G(R)_w = \rho^{-1}(\PGL_2(\bF_2\pbb{y})\times\PGL_2(\bF_2\pbb{t})) \subseteq G(R).
\]
Hence we can embed $G(R)_w$ as a discrete subgroup in $\PGL_2(\bF_2\pbb{y})\times\PGL_2(\bF_2\pbb{t})$, which is compact. This implies that $G(R)_w$ must be finite.
\end{proof}

As a consequence, every $g\in G(R)_w$ must have a finite order. To show that this finite order can be at most $2$, we consider first the group $\PGL_2(\bF_2\prr{y})$ which certainly contains $G(R)$.

\begin{lem}\label{lem:orderinPGL}
If $a\in\PGL_2(\bF_2\prr{y})$ has a finite order, then its order is $1$, $2$ or $3$.
\end{lem}
\begin{proof}
Suppose that $a\in\PGL_2(\bF_2\prr{y})\setminus\{1\}$ has a finite order $r$ and let $A\in\GL_2(\bF_2\prr{y})$ be a choice of lifting. Then its minimal polynomial $m_A(X)$ has degree $2$ and is a divisor of the polynomial $X^r-f$ for some $f\in\bF_2\prr{y}$. Consider first the following two cases:
\begin{itemize}
\item $r=2^e$ for some $e\in\bN$. Let $\tilde{f}\in\overline{\bF_2\prr{y}}$ be such that $f=(\tilde{f})^{2^e}$. Then $X^r-f=(X-\tilde{f})^{2^e}$, implying that $m_A(X)=(X-\tilde{f})^2=X^2-\tilde{f}^2$, i.e.~$a^2=1$ in $\PGL_2(\overline{\bF_2\prr{y}})$. Therefore $r$ must be $2$ in this case.
\item $r$ is an odd number. Then by \cite[Prop.1.1]{pgltwo}, there exists a primitive $r$-th root of unity $\xi\in\overline{\bF_2\prr{y}}$ such that $\xi+\xi^{-1}$ lies in $\bF_2\prr{y}$. Since $\xi$ is then algebraic over $\bF_2$, we have $\xi+\xi^{-1}\in\bF_2\prr{y}\cap\bar{\bF}_2 = \bF_2$, i.e.~$\xi+\xi^{-1}$ is $0$ or $1$. If $\xi+\xi^{-1}=0$, then $\xi=1$, contradicting our assumption on $\xi$. Hence we have $\xi+\xi^{-1}=1$, which implies that $r=3$.
\end{itemize}
In the general case, writing $r=2^en$ with $e,n\in\bN_0$ and $2\nmid n$, we have from the argument above that $e\leq 1$ and $n\in\{1,3\}$. Hence we still have to exclude the case $r=6$. So let $\tilde{f}\in\overline{\bF_2\prr{y}}$ be such that $\tilde{f}^6=f$. Then $L:=\bF_2\prr{y}(\tilde{f})$ is an extension of $\bF_2\prr{y}$ of separable degree at most $3$ and inseparable degree at most $2$, and we have the following prime factorization in $L[X]$:
\[
X^6-f = (X^3-\tilde{f}^3)^2 = (X-\tilde{f})^2(X^2+\tilde{f}X+\tilde{f}^2)^2.
\]
Since $\deg m_A(X) = 2$, we have either $m_A(X)=(X-\tilde{f})^2 = X^2-\tilde{f}^2$ which implies that $a$ has order $2$, or $m_A(X)=X^2+\tilde{f}X+\tilde{f}^2$ is a divisor of $X^3-\tilde{f}^3$ which implies that $a$ has order $3$. Both cases imply that $a$ can't have order $6$, and we are done.
\end{proof}

\begin{lem}\label{lem:orderinquaterniongroup}
The group $G(K)$ has no element of order $3$.
\end{lem}
\begin{proof}
Suppose that $G(K)$ had an element of order $3$ with a lifting $A\in Q$. Then the minimal polynomial of $A$ over $K$, namely $X^2-\tr(A)X+\rn(A)$, must divide $X^3-r$ for some $r\in K$. This implies that $\rn(A)+\tr(A)^2=0$. In particular, $\tr(A)\neq 0$ and $B:=A/\tr(A)\in Q$ is a root of $X^2+X+1$. Hence the quadratic extension $K(\zeta)/K$ mit $\zeta^2+\zeta+1=0$ can be embedded in $Q$, i.e.~$Q$ splits over $K(\zeta)$ by \cite[Ch.I, Thm.2.8]{vigneras}. This contradicts the fact that $Q$ ramifies in $\fp_\zeta$ since $K(\zeta)$ can be embedded in $K_{\zeta}$.
\end{proof}

\begin{prop}\label{prop:stabilizer}
The stabilizer $G(R)_w$ of $w=(w_0,w_\infty)$ in $G(R)$ has only two elements, namely $1$ and the image of $D:=(1+z+z^2)+IJ$ under the projection $\fO^\times\to G(R)$.
\end{prop}
\begin{nota}
The image of $D=(1+z+z^2)+IJ$ under the projection $\fO\to G(R)$ will be denoted by $d\in G(R)$.
\end{nota}
\begin{proof}
First of all, notice that $n(D) = 1+z+z^2 \in R^\times$, so that $D$ is in fact invertible in $\fO$ and its image $d\in G(R)$ is well-defined. To show that $d\in G(R)_w$, observe that
\[
\rho_y(D) = \matzz{1+z+z^2}{(1+z^3)y}{1+y}{1+z+z^2} \equiv \matzz1011 \pmod{y},
\]
implying that $\rho_y(D)\in\GL_2(\bF_2\pbb{y})$, and
\begin{align*}
u \rho_t(D) &= \matzz{1+u+u^2}{(1+u+t)(1+u^3)}{1+t/u}{1+u+u^2} \equiv \matzz1101 \pmod{t},
\end{align*}
implying that $\rho_t(uD)\in\GL_2(\bF_2\pbb{t})$. Hence $\rho(d)\in\PGL_2(\bF_2\pbb{y})\times\PGL_2(\bF_2\pbb{t})$, i.e.~$d\in G(R)_w$.\par\smallskip

Next we have to show that $1,d$ are the only elements in $G(R)_w$. By Lemma \ref{lem:orderinPGL} and \ref{lem:orderinquaterniongroup}, every element in $G(R)_w$ has order $2$, i.e.~$G(R)_w$ is a $2$-elementary abelian group. Hence if $E_1,E_2\in Q$ are liftings of two elements of $G(R)_w$, then $E_1E_2=\lambda E_2E_1$ for some $\lambda\in K^\times$. Taking the reduced norms of both sides, we obtain $\lambda^2=1$, i.e.~$\lambda=1$ and thus $E_1E_2=E_2E_1$.\par\smallskip

Let $L\subseteq Q$ be the $K$-subalgebra generated by the liftings of elements of $G(R)_w$. Then $L$ is commutative, implying that $L/K$ is a field extension of degree $2$. It is purely inseparable since $D\in L$ satisfies $D^2=1+z+z^2$. Hence for all $x\in L$, we have $\rn(x)=x^2$, i.e.
\[
\rn|_{L^\times}:L^\times\longto K^\times, ~ x\longmapsto \rn(x)=x^2
\]
is injective. This also induces the injective map $L^\times/K^\times \hookrightarrow K^\times/(K^\times)^2$, which can be further restricted to the injection
\[
G(R)_w \lhook\joinrel\longto R^\times/(R^\times)^2 = (z)^{\bZ/2\bZ}\times(1+z)^{\bZ/2\bZ}\times(1+z+z^2)^{\bZ/2\bZ} \subseteq K^\times/(K^\times)^2.
\]
On the other hand, by Corollary \ref{cor:diststdvertex}, we must have $\nu_0(\rn(E))\equiv\nu_\infty(\rn(E))\equiv0\pmod2$ for any lifting $E$ of an element of $G(R)_w$. Hence $G(R)_w$ can be embedded at most in the group $(1+z+z^2)^{\bZ/2\bZ}$, which has exactly two elements. Therefore $G(R)_w=\{1,d\}$ as desired.
\end{proof}

\begin{rmk}
Here it is important to consider the subgroup $G(R)$ instead of the whole $G(K)$, since otherwise the stabilizer could have infinitely many elements, or even worse, an element of infinite order. In fact, a computation shows that, for instance, the image of $E:=1+zI+IJ\in Q$ in $G(K)$ has an infinite order but lies in $G(K)_w$.
\end{rmk}

\begin{cor} \label{cor:stabinGR1}
The stabilizer $G(R_1)_w$ in $G(R_1)$ is trivial.
\end{cor}
\begin{proof}
Considering $G(R_1)$ as a subgroup of $G(R)$, we have
\[
G(R_1)_w = G(R_1) \cap G(R)_w \subseteq \{1,d\}.
\]
Since a (and hence any) lifting of $d$ in $Q$ has an odd order at the place $\{z=\zeta\}$, it cannot be invertible over $R_1$. Hence $d\notin G(R_1)$, implying that $G(R_1)_w$ is trivial.
\end{proof}

\subsection{Presentations of arithmetic lattices}
We begin by establishing how the standard vertex is translated to its neighbors.

\begin{lem}\label{lem:neighborsinGR}
For $\dA,\dB$ as in Lemma \ref{lem:V4inGR}, we have
\begin{align*}
\dA.w &= \{(w_0,v) \mid v\in T_\infty ~ \text{is a neighbor of} ~ w_\infty\} ~ \text{and} \\
\dB.w &= \{(v,w_\infty) \mid v\in T_0 ~ \text{is a neighbor of} ~ w_0\}.
\end{align*}
\end{lem}

\begin{proof}
Since the images of $B_1$ and $C_1$ under $\rho_y$ are matrices over $\bF_2\pbb{y}$ and $\rn(B_1)=\rn(C_1)=1+z\in\bF_2\pbb{y}^\times$, we have $B_1.w_0=C_1.w_0=w_0$ by Corollary \ref{cor:diststdvertex}, hence also $B_1^{-1}.w_0=w_0$. This implies that $\dA.w_0=\{w_0\}$. For the vertical component, observe that
\begin{align*}
u\,\rho_t(B_1) &= \matzz{(1+u)(1+u+t)}{u+u^4}{1}{(1+u)(u+t)} ~ \text{and} \\
u\,\rho_t(C_1) &= \matzz{u+u^2}{(1+u^3)(1+u+t)}{(u+t)/u}{u+u^2}.
\end{align*}
Since both matrices on the right hand side have coprime entries in $\bF_2\pbb{t}$ and determinant $u+u^2$, we have by Proposition \ref{cor:diststdvertex} that
\[
d(b_1.w_\infty,w_\infty)=d(c_1.w_\infty,w_\infty)=\nu_0(u+u^2)=1
\]
Since the action of $G(R)$ on $T_\infty$ respects the distance between two vertices on the tree, we also have $d(b_1^{-1}.w_\infty,w_\infty) = d(w_\infty,b_1.w_\infty) = 1$, which implies that
\[
\dA.w \subseteq \{(w_0,v) \mid v\in T_\infty ~ \text{is a neighbor of} ~ w_\infty\}.
\]
Furthermore, $b_1.w,b_1^{-1},c_1.w$ are all different by Corollary \ref{cor:stabinGR1}. Hence $\dA.w$ has exactly $3$ different elements, i.e.~the inclusion above is indeed an equality.\par\smallskip

Now we come to the set $\dB.w$. Since $\rho_y(B_2)$ and $\rho_y(C_2)$ are matrices over with coprime entries $\bF_2\pbb{y}$ and $\rn(B_2)=\rn(C_2)=z+z^2$, we have by Corollary \ref{cor:diststdvertex} that
\[
d(b_2.w_0,w_0) = d(c_2.w_0,w_0) = \nu_0(z+z^2) = 1
\]
From this we can also conclude that $d(b_2^{-1}.w_0,w_0) = d(w_0,b_2.w_0) = 1$, i.e.~$\dB.w_0$ is contained in the subset of the neighbors of $w_0$. For the vertical component, observe that
\begin{align*}
u\,\rho_t(B_2) &= \matzz{(1+u)t}{(1+t)(1+u^3)}{t/u}{(1+u)(1+t)} ~ ~ \text{and} \\
u\,\rho_t(C_2) &= \matzz{1+u^2}{(1+u+t)(1+u^3)}{(u+t)/u}{1+u^2}.
\end{align*}
Both matrices have entries in $\bF_2\pbb{t}$ and determinant $1+u\in\bF_2\pbb{t}^\times$. Hence by Corollary \ref{cor:diststdvertex}, we have $b_2.w_\infty=c_2.w_\infty=w_\infty$, thus also $b_2^{-1}.w_\infty=w_\infty$, i.e. $\dB.w_\infty=\{w_\infty\}$. Therefore,
\[
\dB.w \subseteq \{(v,w_\infty) \mid v\in T_0 ~ \text{is a neighbor of} ~ w_0\}.
\]
The both sets coincide by the same argument applied to $\dA.w$, and we are done.
\end{proof}

Now we can determine presentations of arithmetic lattices by means of Bruhat-Tits action. The result for $G(R_1)$ is pretty simple.

\begin{prop} \label{prop:presentationofLambda}
The group $\Lambda$ coincides with $G(R_1)$ and has the presentation
\begin{equation} \label{eq:presentationofLambda}
\Lambda = \bigl\langle b_1,b_2,c_1,c_2 ~ \big| ~ c_1^2, ~ c_2^2, ~ c_1c_2 = c_2c_1, ~ b_1b_2c_1b_2, ~ b_1c_2b_1b_2^{-1} \bigr\rangle.
\end{equation}
\end{prop}

\begin{proof}
Observe that by Lemmas \ref{lem:V4inGR} and \ref{lem:neighborsinGR}, the pair $(\dA,\dB)$ satisfies the condition of Theorem \ref{thm:compareonTxT}. Hence $\Lambda$ acts on the vertices of $T_0\times T_\infty$ simply transitively and is isomorphic to $\pi_1^\orb(\Sigma_{\dA,\dB},V_4,s_{00})$. Furthermore, by the computation in the proof of Lemma \ref{lem:V4inGR}, there are three $V_4$-orbits of the squares in the square complex $\Sigma_{\dA,\dB}$ with the following representatives:
\[
[b_1,b_2;b_2^{-1},c_1], ~ [b_1,c_2;b_2,b_1^{-1}] ~ \text{and} ~ [c_1,c_2;c_2,c_1].
\]
Therefore, in virtue of Remark \ref{pi1orbSABreduced}, we can use $\{\alpha_{b_1},\alpha_{c_1},\beta_{b_2},\alpha_{c_2}\}$ as a generating system for $\pi_1^\orb(\Sigma_{\dA,\dB},V_4,s_{00})$ and obtain the following presentation:
\begin{align*}
\pi_1^\orb(\Sigma_{\dA,\dB},V_4,s_{00}) &= \left\langle \alpha_{b_1},\alpha_{c_1},\beta_{b_2},\beta_{c_2} ~ \left| ~ 
\begin{array}{c}\alpha_{b_1}\beta_{b_2}=\beta_{b_2}^{-1}\alpha_{c_1}, ~ \alpha_{b_1}\beta_{c_2}=\beta_{b_2}\alpha_{b_1}^{-1} \\
\alpha_{c_1}\beta_{c_2}=\beta_{c_2}\alpha_{c_1}, ~ \alpha_{c_1}^2=\beta_{c_2}^2=1 \end{array} \right. \right\rangle \\
&= \bigl\langle \alpha_{b_1},\alpha_{c_1},\beta_{b_2},\beta_{c_2} \, \big| \, \alpha_{c_1}^2, \, \beta_{c_2}^2, \, \alpha_{c_1}\beta_{c_2}\!=\!\beta_{c_2}\alpha_{c_1}, \, \alpha_{b_1}\beta_{b_2}\alpha_{c_1}\beta_{b_2}, \, \alpha_{b_1}\beta_{c_2}\alpha_{b_1}\beta_{b_2}^{-1} \bigr\rangle.
\end{align*}
The isomorphism from Theorem \ref{thm:compareonTxT} then yields the presentation of $\Lambda$ as in \eqref{eq:presentationofLambda}. Furthermore, since $\Lambda$ acts on the vertices of $T_0\times T_\infty$ transitively and $G(R_1)_w$ is trivial by Corollary \ref{cor:stabinGR1}, we have $G(R_1)=\Lambda$ as desired.
\end{proof}

Now we come to $G(R)$, the largest arithmetic lattice we are studying in this article:

\begin{thm}\label{presentationGR}
The quaternionic arithmetic lattice $G(R)$ has the following presentation:
\[
G(R) = \left\langle b_1,b_2,c_1,c_2,d ~ \left| \begin{array}{c}c_1^2, ~ c_2^2, ~ d^2, ~ c_1c_2 = c_2c_1, ~ c_1d=dc_1, ~ c_2d=dc_2,\\
b_1b_2c_1b_2, ~ b_1c_2b_1b_2^{-1}, ~ db_1db_1, ~ db_2db_2 \end{array} \right. \right\rangle.
\]
\end{thm}
\begin{proof}
Since $\Lambda$ acts on the vertices of $T_0\times T_\infty$ simply transitively, we have
\[
G(R) = \Lambda \cdot G(R)_w.
\]
Since the action of $\Lambda$ on the vertices of $T_0\times T_\infty$ is simple, the intersection $\Lambda\cap G(R)_w$ is trivial. Furthermore, since $G(R)_w=\langle d \rangle$ has exactly two elements, $\Lambda$ is a subgroup of $G(R)$ of index two and thus a normal subgroup. Hence we obtain a presentation of $G(R)$ by adding the relation $d^2=1$ from $G(R)_w$ and the relations obtained by conjugating the generators of $\Lambda$ by $d$ to the relations in the presentation of $\Lambda$ as follows:

\begin{itemize}
\item
Since $C_1,C_2,D$ lie in the commutative subalgebra of $Q$ generated by $IJ$, the elements $c_1$ and $c_2$ are invariant under the conjugation by $d$. These are equivalent to $c_1d=dc_1$ and $c_2d=dc_2$.
\item
For $b_1$, we have $DB_1=J$, implying that $(db_1)^2=1$ since $J^2=1+z^3\in K$. Consequently, we have $db_1d^{-1}=(db_1)^{-1}d^{-1}=b_1^{-1}d^{-2}=b_1^{-1}$, which is equivalent to $db_1db_1=1$.
\item
For $b_2$, we have $DB_2=J+IJ$, implying that $(db_2)^2=1$ since $(J+IJ)^2=z+z^4\in K$. Hence we get $db_2d^{-1}=(db_2)^{-1}d^{-1}=b_2^{-1}d^{-2}=b_2^{-1}$, which is equivalent to $db_2db_2=1$.
\end{itemize}

Combining these relations to those in $\Lambda$ obtained before, we get the presentation of $G(R)$ as desired.
\end{proof}

We also wish to find a torsion-free arithmetic lattice with four orbits of vertices under the Bruhat-Tits action. This corresponds to the fundamental group of $\Sigma_{\dA,\dB}$.

\begin{prop}\label{prop:fundgrpinGR1}
Let $a_1:=c_1b_1^{-1}$ and $a_2:=c_2b_2^{-1}$. Then $\Gamma:=\langle a_1,a_2 \rangle \leq G(R_1)$ is a normal subgroup of index $4$ and isomorphic to the fundamental group $\pi_1(\Sigma_{\dA,\dB},s_{00})$ under the isomorphism $\varphi:\pi_1^\orb(\Sigma_{\dA,\dB},V_4,s_{00})\xrightarrow{\sim} G(R_1)$ from Theorem \ref{thm:compareonTxT}. Its presentation is given by
\begin{equation}\label{eq:presGamma}
\Gamma = \bigl\langle a_1,a_2 ~ \big| ~ a_2a_1^{-1}a_2^2a_1a_2a_1a_2^2a_1^{-1}a_2a_1, ~ a_1a_2^2a_1^{-1}a_2^2a_1^{-1}a_2^2a_1a_2a_1^{-1}a_2 \bigr\rangle.
\end{equation}
\end{prop}

\begin{proof}
This follows from Theorem 1.17. Since $\pi_1(\Sigma_{\dA,\dB},s_{00})$ is the kernel of the surjection
\[
\pi_1^\orb(\Sigma_{\dA,\dB},V_4,s_{00}) \longto V_4, ~ \alpha_a \longmapsto \gamma_v ~ \text{and} ~ \beta_b \longmapsto \gamma_h,
\]
it follows under the isomorphism $\varphi$ that $\Gamma:=\varphi(\pi_1(\Sigma_{\dA,\dB},s_{00}))$ s a normal subgroup of index $4$ in $G(R_1)$. Its presentation can be computed, for instance, by \textsc{Magma}.
\end{proof}

\begin{cor}\label{cor:gammaab}
The abelianization $\Gamma^\ab$ is isomorphic to $\bZ/15\bZ$.
\end{cor}
\begin{proof}
From the presentation in Proposition \ref{prop:fundgrpinGR1}, we get
\[
\Gamma^\ab = \bigl\langle a_1,a_2 ~ \big| ~ a_1+7a_2,8a_2-a_1 \bigr\rangle_\ab.
\]
This means that $\Gamma^\ab$ is isomorphic to the cokernel of
\[
\bZ^2\longto\bZ^2, ~ x\longmapsto Mx, ~ \text{where} ~ M=\matzz1{-1}{7}8.
\]
Since the entries of $M$ are relative prime and $\det(M)=15$, the only elementary divisor of $\Gamma^\ab$ is then $15$. Hence the claim follows.
\end{proof}

\section{Construction of a fake quadric in characteristic $2$} \label{sec:fakequadric}
We are going to construct a \textbf{fake quadric} in characteristic $2$, i.e.~a smooth minimal projective surface of general type $X$ with trivial Albanese variety and the same Chern numbers
\begin{equation*}
\rc_1(X)^2=8 \quad \text{and} \quad \rc_2(X)=4
\end{equation*}
as a quadric surface. The construction is done by means of non-archimedean uniformization, which was employed by Mumford in his construction of a fake projective plane as well as Stix and Vdovina in their construction of a fake quadric in characteristic $3$.

\subsection{The first steps in the construction}
We recapitulate first the construction of the ``wonderful scheme'' given in \cite{stix-vdovina}. In what follows and unlike the previous section, let $R$ be a complete discrete valuation ring with \emph{finite} residue field $k\cong\bF_q$ and $\pi\in R$ be a uniformizer. Furthermore, let $K=R[\frac1\pi]$ be its fraction field. 

\begin{defi} \label{def:wonderfulscheme}
For each $L\in\GL_2(K)$, we define $\ell_i=\ell_i^L\in K[x_0,x_1]$ for $i=0,1$ by
\[
\vekz{\ell_0}{\ell_1} = L^{-1}\vekz{x_0}{x_1}.
\]
Considering $R\bigl[\frac{\ell_0}{\ell_1},\frac{\pi\ell_1}{\ell_0}\bigr]$ as a subring of $K\bigl(\frac{x_0}{x_1}\bigr)$, we define
\[
\tilde{Y}_L := \Spec R\left[\frac{\ell_0}{\ell_1},\frac{\pi\ell_1}{\ell_0}\right] \cong \Spec R[u,v]/(uv-\pi).
\]
This is a regular scheme of finite type over $K$ and only depends on the image of $L$ in $\PGL_2(K)$. Its special fiber consists of two irreducible components isomorphic to $\bP_k^1$ which intersect transversely, while its generic fiber is a complement of two $K$-rational points on a projective line $\Proj K[\ell_0,\ell_1] \cong \bP^1_K$. Furthermore, by excluding all the finitely many $k$-rational (closed) points on the special fiber of $\tilde{Y}_L$ except the singular point $(0,0)$, we obtain the open subscheme
\[
Y_L \subset \tilde{Y}_L.
\]
Since all schemes in the family $(Y_L)_{L\in\PGL_2(K)}$ have the same function field $K\bigl(\frac{x_0}{x_1}\bigr)$, they can all be glued to the seperated $R$-scheme
\[
Y = \bigcup_{L\in\PGL_2(K)} Y_L.
\]
with $K\bigl(\frac{x_0}{x_1}\bigr)$ as function field. Its generic fiber $Y_K=Y\otimes_RK$ is isomorphic to $\bP_K^1 = \Proj K[x_0,x_1]$. To describe the special fiber $Y_s = Y\otimes_Rk$, observe first that the quotient homomorphism $R[u,v]/(uv-\pi)\surj R[u,v]/(uv-\pi,v,\pi)\cong k[u]$ induces a rational map
\[
\bA_k^1 = \Spec k[u] \longto \Spec R[u,v]/(uv-\pi) \cong \tilde{Y}_L \dashrightarrow Y_L \longhookrightarrow Y
\]
for each $L\in\PGL_2(K)$. The closure of its image in $Y$ is isomorphic to the projective line $\bP_k^1$ and depends on $L$ up to its class in
\[
\GL_2(K)/(K^\times\cdot\GL_2(R)) = \PGL_2(K)/\PGL_2(R).
\]
\end{defi}

The next step is to define a group action of $\PGL_2(K)$ on the scheme $Y$. In what follows, let $x$ denote the column vector $(x_0,x_1)^t$ consisting of the variables $x_0,x_1\in K[x_0,x_1]$.

\begin{defi}\label{def:actiononcoord}
The left action of $\GL_2(K)$ on $K[x_0,x_1]$ is defined on the variables $x_0,x_1$ by
\begin{equation*}
S^*(x) = (S^*(x_0),S^*(x_1))^t = S^{-1}x \quad \text{for} ~ S\in\GL_2(K).
\end{equation*}
\end{defi}

\begin{defi}
The action of $\PGL_2(K)$ on $Y$ is defined as follows: For $L\in \PGL_2(K)$ and $S\in\GL_2(K)$, we have
\begin{equation*}
S^*\vekz{\ell_0^L}{\ell_1^L} = S^*(L^{-1}x) = L^{-1}(S^*(x)) = L^{-1}(S^{-1}x) = (SL)^{-1}(x).
\end{equation*}
This defines an isomorphism from $R\left[\frac{\ell_0^L}{\ell_1^L},\frac{\pi\ell_1^L}{\ell_0^L}\right]$ to $R\left[\frac{\ell_0^{SL}}{\ell_1^{SL}},\frac{\pi\ell_1^{SL}}{\ell_0^{SL}}\right]$. Its inverse yields an isomorphism
\[
\sigma_{S,L} : Y_L\to Y_{SL}.
\]
Gluing the isomorphisms $\sigma_{S,L}$ for all $L\in\PGL_2(K)$ together, we obtain the isomorphism
\[
\sigma_S : Y\to Y.
\]
Since this doesn't depend on scalar multiples of $S$, we obtain the group action
\[
\PGL_2(K) \longto \Aut_R(Y), \quad S \longmapsto \sigma_S.
\]
\end{defi}

Note that this group action can be restricted to the group actions on the generic fiber and on the special fiber respectively. On the generic fiber, the action of $S\in\PGL_2(K)$ is given by
\[
\sigma_S|_{\bP_K^1}(x) = Sx
\]
for homogeneous coordinates $x_0,x_1$ with $x=\left(\begin{smallmatrix} x_0\\x_1 \end{smallmatrix}\right)$. To describe the action on the special fiber $Y_s$, we consider instead the dual graph of $Y_s$. This consists of a vertex $v_C$ for each irreducible component $C$ of $Y_s$. Two vertices are joined if and only if the corresponding components intersect in (necessarily exactly) one double point.

\begin{lem} \label{dualgraphvstree}
There is a $\PGL_2(K)$-equivariant bijection between the following sets:
\begin{enumerate}
\item
Irreducible components of the special fiber $Y_s$.
\item
Homothety classes of $R$-lattices in $H^0(\bP_K^1,\cO(1)) = K\cdot x_0 \oplus K\cdot x_1$ under the $\PGL_2(K)$-action defined on the coordinates $(x_0,x_1)$ by Definition \ref{def:actiononcoord}.
\item
Vertices of the Bruhat-Tits tree $T_K=\Delta(\PGL_2(K))$.
\end{enumerate}
In particular, there is a $\PGL_2(K)$-equivariant isomorphism between the dual graph of the special fiber $Y_s$ and $T_K$.
\end{lem}

\begin{proof}
As discussed at the end of Definition \ref{def:wonderfulscheme}, each irreducible component of $Y_s$ is given by a class in $\PGL_2(K)/\PGL_2(R)$, thus corresponds to the homothety class of the lattice
\[
M_L = R\cdot\ell_0^L \oplus R\cdot\ell_1^L = H^0(P_L,\cO(1)) \subseteq H^0(\bP_K^1, \cO(1)).
\]
This 1-1 correspondence is $\PGL_2(K)$-equivariant by definition. The $\PGL_2(K)$-equivariant and inclusion-reversing bijection between the $R$-lattices of $H^0(\bP_K,\cO(1))$ and those of $K^2$ is given by
\[
M \subseteq H^0(\bP_K,\cO(1)) ~ \longmapsto ~ \bigl\{ a=(a_0,a_1)^t\in K^2 \, \big| \, \forall f\in M: f(a) \in R \bigr\}.
\]

To see that this induces a $\PGL_2(K)$-equivariant isomorphism between the dual graph of the special fiber $Y_s$ and $T_K$, observe that two irreducible components $C_1,C_2$ of $Y_s$ intersect at a point $P$ if and only if $Y_L$ is a neighborhood of $P$ and the closure of its special fiber is the union of $C_1$ and $C_2$. This holds if and only if the corresponding lattices $M_1,M_2$ are of the form
\[
M_1 = R\cdot\ell_0^L \oplus R\cdot\ell_1^L \quad \text{and} \quad M_2 = R\cdot\ell_1^L \oplus R\cdot\pi\ell_0^L.
\]
This is the case if and only if $\pi M_1\subsetneq M_2 \subsetneq M_1$, i.e.~the corresponding vertices in $T_K$ are joined by an edge. Hence we obtain a $\PGL_2(K)$-equivariant isomorphism between the dual graph of $Y_s$ and the Bruhat-Tits tree $T_K$ as desired.
\end{proof}

Now consider the product $Y\times_RY$. It is locally given by an open subscheme of the affine spectrum of
\[
R[u,v]/(uv-\pi) \otimes_R R[w,z]/(wz-\pi) = R[u,v,w,z]/(uv-\pi, wz-\pi)
\]
which contains the point $(u,v,w,z)=(0,0,0,0)$. Since $R[u,v]/(uv-\pi)$ is regular and smooth over $R$ outside of $(0,0)$, the scheme $Y\times_RY$ is regular up to the points corresponding to $(u,v,w,z)=(0,0,0,0)$ in the local chart given before. There $u,v,w+z$ forms a regular sequence of length $3=\dim Y\times_RY$. Hence $Y\times_RY$ is Cohen-Macaulay and normal. The dual complex $\Sigma$ of its special fiber $(Y\times_R Y)_s = (Y\times_R Y)\otimes_Rk = Y_s\times_kY_s$ can be described as follows:
\begin{enumerate}
\item
Vertices of $\Sigma$ are irreducible components of $(Y\times_R Y)_s$. These are isomorphic to $\bP_k^1\times\bP_k^1$.
\item
Unoriented edges of $\Sigma$ are given by irreducible curves in the intersection of two irreducible components of $(Y\times_R Y)_s$. Each edge is attached to the vertices given by the corresponding irreducible components. Here an irreducible component can, after identification with $\bP_k^1\times\bP_k^1$, intersect another component only on the grid lines
\[
(\bP_k^1(k)\times\bP_k^1)\cup(\bP_k^1\times\bP_k^1(k)) \subseteq \bP_k^1\times_k\bP_k^1
\]
and the intersection yields exactly a projective line $\bP_k^1$.
\item
Squares of $\Sigma$ are given by singular points of $(Y\times_R Y)_s$, i.e.~by $\bP^1(k)\times\bP^1(k)$ on each irreducible component $\bP_k^1\times\bP_k^1$. At each such point $P$, the scheme $Y\times_RY$ is locally isomorphic to $\Spec(R[u,v,w,z]/(uv=\pi=wz))$ with $P$ mapping to $(0,0,0,0)$. Hence on the special fiber, $P$ lies on exactly four irreducible components locally given by
\[
\{u=w=0\}, ~ \{u=z=0\}, ~ \{v=z=0\}, ~ \text{and} ~ \{v=w=0\}.
\]
By considering how each two out of these components intersect, we see that the $2$-cell corresponding to $P$ is indeed a square.
\end{enumerate}
The action of $\PGL_2(K)$ on $Y$ then induces the action of $\PGL_2(K)\times\PGL_2(K)$ on $Y\times_RY$. Its restriction on the special fiber $(Y\times_R Y)_s$ can then be described as in the following lemma:

\begin{lem}
There is a $\PGL_2(K)\times\PGL_2(K)$-equivariant isomorphism between the dual complex of $(Y\times_R Y)_s$ with the action described above and the product of Bruhat-Tits tree $T_K\times T_K$ with Bruhat-Tits-action.
\end{lem}
\begin{proof}
This follows immediately from Lemma \ref{dualgraphvstree} since the product of the dual graph of $Y_s$ with itself is exactly the dual complex $\Sigma$.
\end{proof}

\subsection{The formal scheme and its quotient}
Now we wish to build a quotient of $Y\times_R Y$ by an arithmetic subgroup of $\PGL_2(K)\times\PGL_2(K)$. One problem is that the $K$-rational points on the generic fiber $Y_K$ are also closed points in $Y$. To avoid this problem, we consider instead its formal completion as follows:

\begin{defi} \label{def:formalwonderful}
The formal scheme $\cY/\Spf(R)$ is defined as the formal completion of the scheme $Y/R$ along its special fiber $Y_s$.
\end{defi}

\begin{rmk}
The formal scheme $\cY$ has the following properties:
\begin{enumerate}
\item
Its generic fiber in the sense of Raynaud's rigid analytic geometry is isomorphic to the Drinfeld upper half plane $\Omega_K^1$, the analytification of the complement of $\bP^1(K)$ in $\bP_K^1$.
\item
It is also possible to construct $\cY$ by gluing as follows: For each $L\in\PGL_2(K)$, let $\hat{Y}_L$ denote the formal completion of $Y_L$ along its special fiber. By the functoriality of the gluing construction of $Y$, we can glue all $\hat{Y}_L$'s together and obtain
\[
\cY = \bigcup_{L\in\PGL_2(K)}\hat{Y}_L.
\]
\item
The formal scheme $\cY$ also carries a $\PGL_2(K)$-action obtained from the corresponding action on $Y$. Each $S\in\GL_2(K)$ sends the open subscheme $\hat{Y}_L$ for $L\in\PGL_2(K)$ to $\hat{Y}_{SL}$.
\item
The special fiber $\cY_s$ is isomorphic to $Y_s$. Its dual graph can be identified with the Bruhat-Tits tree $T_K$ in a $\PGL_2(K)$-equivariant way according to Lemma \ref{dualgraphvstree}.
\item
The product $\cY\times_R\cY$ over $\Spf(R)$ can also be obtained by completing of $Y\times_RY$ along its special fiber. The generic fiber in the sense of rigid geometry is $\Omega_K^1\times\Omega_K^1$. The action of $\PGL_2(K)\times\PGL_2(K)$ on $Y\times_RY$ extends to its completion. Its special fiber is the same as $(Y\times_RY)_s$ and hence has the same dual complex $T_K\times T_K$ as before.
\end{enumerate}
\end{rmk}

Now let $\Gamma\leq\PGL_2(K)\times\PGL_2(K)$ be a discrete torsion-free subgroup acting cocompactly on $T_K\times T_K$ via Bruhat-Tits action. This induces a free and discontinuous action on the special fiber $(Y\times_RY)_s$ with respect to Zariski topology, thus also on the formal scheme $\cY\times_R\cY$ over $\Spf(R)$. Hence it is possible to build the quotient
\[
\cX_\Gamma := \Gamma\backslash(\cY\times_R\cY).
\]
The dual complex $\Sigma_\Gamma$ of $\cX_\Gamma$ is then the finite quotient complex $\Gamma\backslash\Sigma$. Hence the quotient $\cX_\Gamma$ is proper over $\Spf(R)$. Moreover, we can assert the following fact about the quotient complex:
\begin{lem}
Suppose that the quotient square complex $\Sigma_\Gamma=\Gamma\backslash\Sigma$ has $N$ vertices, then
\[
\#[\bE(\Sigma_\Gamma)] = N(q+1) ~ \text{and} ~ \#\bS(\Sigma_\Gamma) = \frac14N(q+1)^2.
\]
Hence the topological Euler characteristic is
\[
\chi(\Sigma_\Gamma) = \frac14N(q-1)^2.
\]
\end{lem}
\begin{proof}
\cite[Lemma 47]{stix-vdovina}.
\end{proof}

The next step is to transfer our quotient formal scheme back to an algebraic scheme and to study its properties.

\begin{prop}\label{prop:algebraization}
The formal scheme $\cX_\Gamma$ over $\Spf(R)$ is a formal completion along the special fiber of a projective scheme $X_\Gamma$ over $R$. Its generic fiber $X_{\Gamma,K}=X_{\Gamma}\otimes_RK$ is smooth projective with ample canonical line bundle. In particular, it is a minimal surface of general type without smooth rational curves with self-intersection number $-1$ or $-2$.
\end{prop}
\begin{proof}
This is done in \cite[Prop.48]{stix-vdovina}. Indeed, the sheaf of relative log-differentials $\Omega^{2,\log}_{\cX_\Gamma/R}$ obtained by descending the pull back $\Omega^{2,\log}_{Y\times_RY/R}|_{\cY\times_R\cY}$ on $\cY\times_R\cY$ to the quotient $\cX_\Gamma$ is an ample line bundle. Since $\cX_\Gamma$ is proper over $\Spf(R)$, it can be embedded by the ample line bundle $\Omega^{2,\log}_{\cX_\Gamma/R}$ as closed subspace in a formal projective scheme over $\Spf(R)$ and can thus be algebraized to a projective scheme $X_\Gamma/R$ by Grothendieck's formal GAGA principle. Then it has been shown in loc.cit.~that its generic fiber $X_{\Gamma,K}/K$ is smooth and the canonical sheaf $\omega_{X_{\Gamma,K}/K}=\Omega^{2,\log}_{X_\Gamma/R}|_{X_{\Gamma,K}}$ is ample. The last statement follows from the adjunction formula.
\end{proof}

\subsection{Computing the numerical invariants}
We would like to show that $X_{\Gamma,K}$ as above is a fake quadric for an appropriate choice of $K$ and $\Gamma$, and begin with the following notation:

\begin{nota}
We define the following maps:
\begin{itemize}
\item
For each vertex $E\in \bV(\Sigma_\Gamma)$, let $\pi_E:\tilde{E}\to E$ be the normalization of the corresponding irreducible component $E\subseteq X_{\Gamma,s}$. In this case we have $\tilde{E}\cong \bP_k^1\times_k\bP_k^1$.
\item
For each unoriented edge $C\in [\bE(\Sigma_\Gamma)]$, let $\pi_C:\tilde{C}\to C$ be the normalization of the corresponding irredubible curve on $X_{\Gamma,s}$. In this case we have $\tilde{C}\cong\bP_k^1$.
\item
For each square $P\in \bS(\Sigma_\Gamma)$, let $\iota_P:P\hookrightarrow X_{\Gamma,s}$ be the corresponding $k$-rational point.
\end{itemize}
\end{nota}

Now choosing an orientation on $\Sigma(\cX_\Gamma)$, we obtain the following cellular cochain complex

\begin{equation} \label{cplx:cellularcochain}
0 \longto \bigoplus_{E\in \bV(\Sigma(\cX_\Gamma))}\bZ \longto \bigoplus_{C\in [\bE(\Sigma(\cX_\Gamma))]}\bZ \longto \bigoplus_{P\in \bS(\Sigma(\cX_\Gamma))}\bZ.
\end{equation}

\begin{prop}\label{prop:Chern}
Let $N$ be the number of vertices of $\Sigma(\cX_\Gamma)$. Then we have
\[
\chi(X_{\Gamma,K})=\frac14N(q-1)^2, \quad \rc_1(X_{\Gamma,K})^2 = 2N(q-1)^2 \quad \text{and} \quad \rc_2(X_{\Gamma,K}) = N(q-1)^2.
\]
\end{prop}
\begin{proof}
By flatness of $X_\Gamma$ over $R$, we can compute the Euler characteristic of its generic fiber $X_{\Gamma,K}$ via its special fiber $X_{\Gamma,s}$ using the exact sequence
\[
0 \longto \cO_{X_{\Gamma,s}} \longto \bigoplus_{E\in \bV(\Sigma(\cX_\Gamma))} \pi_{E,\ast}\cO_{\tilde{E}} \longto \bigoplus_{C\in [\bE(\Sigma(\cX_\Gamma))]} \pi_{C,\ast}\cO_{\tilde{C}} \longto \bigoplus_{P\in \bS(\Sigma(\cX_\Gamma))} \iota_{P,\ast}\cO_{P} \longto 0,
\]
where the maps are given by the sum of restrictions with signs coming from \eqref{cplx:cellularcochain}, see \cite[Prop.49]{stix-vdovina}. To compute the first Chern number, observe first that $\omega_{X_{\Gamma,K}/K}=\Omega^{2,\log}_{X_\Gamma/R}|_{X_{\Gamma,K}}$ by Proposition \ref{prop:algebraization}. Hence by flatness of $X_\Gamma$ over $R$, we have
\[
\rc_1(X_{\Gamma,K})^2 = \bigl(\rc_1(\Omega^{2,\log}_{\cX_\Gamma/R})|_{X_{\Gamma,s}}\bigr)^2.
\]
Consequently, $\chi(X_{\Gamma,K})$ and $\rc_1(X_{\Gamma,K})^2$ can be computed as in the computation after loc.cit.~and $\rc_2(X_{\Gamma,K})$ then by Noether's formula.
\end{proof}

To compute the Albenese variety, we endow first $X_\Gamma/R$ with the fs-log structure in the sense of Fontaine and Illusie determined by its special fiber. The resulting log-scheme is projective and log-smooth over $\Spec(R)$ since the special fiber consists of strictly normal crossing divisors which determine all the singularities of $X_\Gamma$. Hence by considering the log-geometric fiber $X_{\Gamma,\tilde{s}}$, where $\tilde{s}\to\Spec(R)$ denotes the log-geometric point over the closed point, we obtain an isomorphism
\[
H^1_\et(X_{\Gamma,\bar{K}},\Lambda) \cong H^1_\ket(X_{\Gamma,\tilde{s}},\Lambda)
\]
from the Kummer-\'etale cospecialization map. Here $X_{\Gamma,\bar{K}}$ denotes the geometric generic fiber over an algebraic closure $\bar{K}$ of $K$ and $\Lambda$ stands for a finite commutative ring of order prime to $p$. The group $H^1_\ket(X_{\Gamma,\tilde{s}},\Lambda)$ can then be computed as follows:

\begin{lem}\label{lem:hypercohom}
Let $\partial$ be sum of the restriction maps
\[
\bigoplus_{E\in \bV(\Sigma(\cX_\Gamma))}H^1_{\ket}(\tilde{E}_{\tilde{s}},\Lambda) \xrightarrow{\pm\res} \bigoplus_{C\in [\bE(\Sigma(\cX_\Gamma))]} H^1_{\ket}(\tilde{C}_{\tilde{s}},\Lambda)
\]
with signs coming from the corresponding cellular cochain complex in \eqref{cplx:cellularcochain}. Then the following sequence is exact:
\begin{equation} \label{exseq:fiveterm}
0 \longto \Hom(\Gamma^\ab,\Lambda) \longto H^1_{\ket}(X_{\Gamma,\tilde{s}},\Lambda) \longto \ker(\partial).
\end{equation}
\end{lem}

\begin{proof}
Observe first that the constant sheaf $\Lambda$ on the Kummer \'etale site $(X_{\Gamma,\tilde{s}})_{\ket}$ is quasi-isomorphic to the complex
\[
0 \longto \!\!\!\bigoplus_{E\in \bV(\Sigma(\cX_\Gamma))}\!\!\! \pi_{E,*}\Lambda \xrightarrow{\pm\res} \!\!\!\bigoplus_{C\in [\bE(\Sigma(\cX_\Gamma))]}\!\!\! \pi_{C,*}\Lambda \xrightarrow{\pm\res} \!\!\!\bigoplus_{P\in \bS(\Sigma(\cX_\Gamma))}\!\!\! \iota_{P,*}\Lambda \longto 0
\]
with signs coming from the associated cochain complex for $\Sigma(\cX_\Gamma)$. Then the exact sequence \eqref{exseq:fiveterm} arises from the five-term exact sequence of the resulting hypercohomology spectral sequence by the fact that $H_{\rm Sing}^1(\Sigma(\cX_\Gamma),\Lambda)\cong \Hom(\Gamma^\ab,\Lambda)$.
\end{proof}

From now on we restrict to the case $K=\bF_2\prr{z}$ and $\Gamma$ is the arithmetic lattice from Proposition \ref{prop:fundgrpinGR1}.

\begin{lem}\label{lem:preparealb}
If $6$ is invertible in $\Lambda$, then the kernel of $\partial$ from Lemma \ref{lem:hypercohom} is trivial.
\end{lem}

\begin{proof}
In what follows, let $\Lambda(-1)$ denote the inverse Tate-twist of $\Lambda$ and $\Maps(-,-)^0$ denote the kernel of the map defined by summing the images over the domain. For each $E\in\bV(\Sigma(\cX_\Gamma))$ with the singular part $E_{\Sing}\subseteq E$, we have $E\setminus E_{\Sing}\cong(\bP_k^1\setminus\bP^1(k))\times(\bP_k^1\setminus\bP^1(k))$, so that
\[
H^1_{\ket}(\tilde{E}_{\tilde{s}},\Lambda) \cong H^1_\et(E\setminus E_{\Sing},\Lambda) \cong H^1_\et(\bP_k^1\setminus\bP^1(k),\Lambda)\oplus H^1_\et(\bP_k^1\setminus\bP^1(k),\Lambda).
\]
Here each summand in the direct sum above comes from the vertical and horizontal component and will be thus denoted in what follows by $H^1(E)_v$ and $H^1(E)_h$ respectively. We can compute $H^1_\et(\bP_k^1\setminus\bP^1(k),\Lambda)$ by the excision principle from the exact sequence
\[
0=H^1_\et(\bP_k^1,\Lambda) \longrightarrow H^1_\et(\bP_k^1\setminus\bP^1(k),\Lambda) \longrightarrow H^2_{\et,\bP^1(k)}(\bP^1,\Lambda) \stackrel{f}{\longrightarrow} H^2_\et(\bP^1_k,\Lambda) = \Lambda(-1).
\]
Under the isomorphism $H^2_{\et,\bP^1(k)}(\bP^1,\Lambda)$ with $H^0_{\et}(\bP^1(k),\Lambda(-1)) = \Maps(\bP^1(k),\Lambda(-1))$, the map $f$ is summation map. This implies that
\[
H^1_\et(\bP_k^1\setminus\bP^1(k),\Lambda) \cong \Maps(\bP^1(k),\Lambda(-1))^0.
\]
Moreover, $\bP^1(k)$ can be interpreted as the set of vertical unoriented edges attached to $E$ in the case of $H^1(E)_h$ and horizontal ones in the case of $H^1(E)_v$, i.e.~identified with $\dA$ and $\dB$ respectively. Hence we get
\begin{gather*}
H^1(E)_h \cong \Maps(\dA,\Lambda(-1))^0 \quad \text{and} \quad H^1(E)_v \cong \Maps(\dB,\Lambda(-1))^0.
\end{gather*}
Similarly, for each $C\in[\bE(\Sigma(\cX_\Gamma))]$, we have
\[
H^1_{\ket}(\tilde{C}_{\tilde{s}},\Lambda) = \Maps(\bS_C,\Lambda(-1))^0,
\]
where $\bS_C\subseteq \bS(\Sigma(\cX_\Gamma))$ denotes the set of the squares in $\Sigma(\cX_\Gamma)$ attached to the edge $C$.

To compute the kernel of $\partial$ explicitly, we choose an orientation on $\Sigma(\cX_\Gamma)$ in such a way that
\begin{equation*}
\partial = \sum_{i,j\in\{0,1\}} \left( (-1)^j\sum_{a\in\dA}\res^{E_{ij}}_{C_{(a,i)}} + (-1)^i\sum_{b\in\dB}\res^{E_{ij}}_{C_{(b,j)}} \right),
\end{equation*}
where each $E_{ij}$ is the irreducible component of $X_{\Gamma,s}$ corresponding to $s_{ij}\in\bV(\Sigma(\cX_\Gamma))$; $C_{(a,i)}$ and $C_{(b,j)}$ the irreducible curve corresponding to $[(a,i)] \in [\bE(\Sigma(\cX_\Gamma))]_v$ and $[(b,j)] \in [\bE(\Sigma(\cX_\Gamma))]_h$ respectively. Then with notation from Subsection \ref{subsec:localpermgroup}, if
\begin{equation*}
\xi=(\xi_{ij}^h,\xi_{ij}^v)_{i,j\in\{0,1\}} \in \bigoplus_{i,j\in\{0,1\}}H^1_{\ket}(\tilde{E}_{ij,\tilde{s}},\Lambda) = \bigoplus_{i,j\in\{0,1\}}\bigl(\Maps(\dA,\Lambda(-1))^0 \oplus \Maps(\dB,\Lambda(-1))^0\bigr),
\end{equation*}
we obtain
\begin{equation} \label{eq:kerdel}
\partial(\xi) = ((\xi_{0j}^h\circ t^0_{(b,j)} - \xi_{1j}^h\circ t^1_{(b,j)})_{(b,j)\in\dB\times\{0,1\}},(\xi_{i0}^v\circ t^0_{(a,i)} - \xi_{i1}^v\circ t^1_{(a,i)})_{(a,i)\in\dA\times\{0,1\}}).
\end{equation}
So let $P^\dA_j$ and $P^\dB_i$ act from the right on $H^1(E_{0j})_h\oplus H^1(E_{1j})_h \subseteq \Maps(\dA\times I, \Lambda(-1))$ and $H^1(E_{i0})_v\oplus H^1(E_{i1})_v \subseteq \Maps(\dB\times I, \Lambda(-1))$ respectively. It follows from \eqref{eq:kerdel} that
\[
\ker(\partial) = \left\{\begin{array}{l}H^0\bigl(P^\dA_0,H^1(E_{00})_h\oplus H^1(E_{10})_h\bigr) \oplus H^0\bigl(P^\dA_1,H^1(E_{01})_h\oplus H^1(E_{11})_h\bigr) \\ 
\oplus{}\, H^0\bigl(P^\dB_0,H^1(E_{00})_v\oplus H^1(E_{01})_v\bigr) \oplus H^0\bigl(P^\dB_0,H^1(E_{10})_v\oplus H^1(E_{11})_v\bigr)\end{array}\right. ~ .
\]

Since $P^\dA_0=P^\dA_1=P^\dA$ and $P^\dB_0=P^\dB_1=P^\dB$ by Proposition \ref{prop:localpermgrpforquarternion} and $P^\dA$ resp.~$P^\dB$ apparently acts on $\dA\times I$ resp.~$\dB\times I$ transitively, we see that $\xi\in\ker\partial$ if and only if $\xi_{00}^h=\xi_{10}^h$, $\xi_{01}^h=\xi_{11}^h$, $\xi_{00}^v=\xi_{01}^v$, $\xi_{10}^v=\xi_{11}^v$ and all of these are constant. On the other hand, the sum of the images of each map $\xi_{ij}^h$ and $\xi_{ij}^v$ must be zero and $3=\#\dA=\#\dB$ is invertible in $\Lambda$. Hence $\xi\in\ker(\partial)$ only if all of its components vanish. In particular, $\ker(\partial)$ must be trivial as desired.
\end{proof}

\begin{prop}\label{prop:trivialalb}
The Albanese variety $\Alb_{X_{\Gamma,K}}$ is trivial.
\end{prop}

\begin{proof}
From Lemmas \ref{lem:hypercohom} and \ref{lem:preparealb}, we have
\[
\Hom(\Gamma^\ab,\bZ/\ell\bZ) \cong H^1_{\ket}(X_{\Gamma,\tilde{s}},\bZ/\ell\bZ) \cong H^1_\et(X_{\Gamma,\bar{K}},\bZ/\ell\bZ)
\]
for any prime $\ell\neq2,3$. Since the maximal abelian quotient of $\Gamma^\ab$ is $\bZ/15\bZ$, compare Corollary \ref{cor:gammaab}, the latter group is trivial for all $\ell\neq3,5$. It follows that
\[
\Hom(\pi_1^{\ab,(\ell)}(X_{\Gamma,\bar{K}}),\bZ/\ell\bZ) \cong H^1_\et(X_{\Gamma,\bar{K}},\bZ/\ell\bZ) = 0
\]
for any prime $\ell\neq2,3,5$. Therefore $\pi_1^{\ab,(\ell)}(X_{\Gamma,\bar{K}})$ contains no free $\bZ_\ell$-submodule, which implies that the Tate module $T_\ell(\Alb_{X_{\Gamma,K}})$ must vanish, whence also $\Alb_{X_{\Gamma,K}}$.
\end{proof}

Having computed the necessary invariants for a fake quadric, we summarize the result in the following theorem.

\begin{thm}\label{thm:fakequadricinchar2}
Let $R=\bF_2\pbb{t}$, $K=\Quot(R)=\bF_2\prr{t}$, and $\Gamma\leq\PGL_2(K)\times\PGL_2(K)$ be the lattice from Proposition \ref{prop:fundgrpinGR1}. Then the surface $X_{\Gamma,K}$ is a fake quadric over $K$.
\end{thm}

\begin{proof}
By Proposition \ref{prop:algebraization}, the surface $X_{\Gamma,K}$ is smooth and projective over $K$ and minimal with an ample canonical bundle. The last property implies that $X_{\Gamma,K}$ is a surface of general type. Since the dual square complex $\Sigma(\cX_\Gamma)$ has four vertices, it follows from Proposition \ref{prop:Chern} that
\[
\rc_1(X_{\Gamma,K})^2 = 2\cdot4\cdot(2-1)^2 = 8 \quad \text{and} \quad \rc_2(X_{\Gamma,K}) = 4\cdot(2-1)^2 = 4.
\]
Furthermore, its Albanese variety is trivial by Proposition \ref{prop:trivialalb}. Therefore, $X_{\Gamma,K}$ is a fake quadric over $K=\bF_2\prr{t}$ as desired.
\end{proof}

\begin{rmk}
Unlike the fake quadric in characteristic $3$ constructed in \cite[Thm.52]{stix-vdovina}, the maximal abelian quotient of $\Gamma$ has order $15$ which is prime to the characteristic of $K$. Hence it is not possible to show that $X_{\Gamma,K}$ has non-reduced Picard scheme by the same method as loc.cit. Indeed, it is still an open question whether the fake quadric $X_{\Gamma,K}$ we constructed here has reduced Picard scheme or not.
\end{rmk}
\bibliographystyle{amsalpha}
\bibliography{bibliography-fakequadric}
\end{document}